\documentclass[a4paper,11pt]{article}
\usepackage{graphics,color}
\usepackage{amsmath}
\usepackage{amssymb}
\usepackage{mathrsfs}
\usepackage{cite}
\usepackage{verbatim}
\usepackage{float}
\usepackage{graphicx}
\usepackage{amsthm}
\usepackage{textcomp}
\usepackage{subfig}

\usepackage{hyperref}

\newcommand{\red}{\color{red}}

\numberwithin{equation}{section}
\mathchardef\emptyset="001F

\newtheorem{Theorem}{Theorem}[section]
\newtheorem{Definition}[Theorem]{Definition}
\newtheorem{Proposition}[Theorem]{Proposition}
\newtheorem{Corollary}[Theorem]{Corollary}
\newtheorem{Lemma}[Theorem]{Lemma}

\newcommand{\nada}[1]{}

\newcommand{\cardsing}{m}

\newcommand{\eps}{\varepsilon}
\newcommand{\grad}{\nabla}
\newcommand{\grado}{{\rm deg}}

\newcommand{\lift}{\Phi}
\newcommand{\mappa}{u}
\newcommand{\mappabuona}{v}
\newcommand{\mres}{\mathbin{\vrule height 1.6ex depth 0pt width
0.13ex\vrule height 0.13ex depth 0pt width 1.3ex}}

\newcommand{\numberset}{\mathbb}
\newcommand{\N}{\numberset{N}} 
\newcommand{\Om}{\Omega} 
\newcommand{\R}{\numberset{R}} 
\newcommand{\Rn}{\numberset{R}^n}

\newcommand{\radius}{\ell}

\newcommand{\sliceTV}{\psi}
\newcommand{\sourceballrx}{B_r(x)}

\newcommand{\Suno}{{\mathbb S}^1}
\newcommand{\totvarjac}{TV\!J}
\newcommand{\triplemap}{u_T}
\newcommand{\vortexmap}{u_V}
\newcommand{\Z}{\numberset{Z}}

\theoremstyle{definition}
\newtheorem{Remark}[Theorem]{Remark}

 \title{The relaxed area of $\Suno$-valued singular maps
in the strict $BV$-convergence 
}

\author{
Giovanni Bellettini\footnote{
Dipartimento di Ingegneria dell'Informazione e Scienze Matematiche, Universit\`a di Siena, 53100 Siena, Italy,
and International Centre for Theoretical Physics ICTP,
Mathematics Section, 34151 Trieste, Italy.
E-mail: bellettini@diism.unisi.it
                      }\and
Simone Carano\footnote{
Area of Mathematical Analysis, Modelling, and Applications,
             Scuola Internazionale Superiore di Studi Avanzati ``SISSA'',
Via Bonomea, 265 - 34136 Trieste, Italy. E-mail: scarano@sissa.it
                         }
                          \and
Riccardo Scala\footnote{
Dipartimento di Ingegneria dell'Informazione e Scienze Matematiche, Universit\`a di Siena, 53100 Siena, Italy.
E-mail: riccardo.scala@unisi.it
}
}

\begin{document}
\maketitle

\begin{abstract}
	Given a bounded open set
$\Om \subset \R^2$, 
we study the relaxation  
of the nonparametric area functional in the strict topology in 
$BV(\Om;\R^2)$, 
and compute it for vortex-type maps, and more generally 
for maps  in $W^{1,1}(\Omega;\Suno)$ having a finite 
number of topological singularities. 
We also 
extend the analysis to some specific piecewise constant maps
in $BV(\Omega;\Suno)$, including the 
symmetric triple junction map.
\end{abstract}


\vspace{2mm}

\noindent {\bf AMS (MOS) 2020 subject clas\-si\-fi\-ca\-tion:}  49J45, 49Q05, 49Q15, 28A75.

\section{Introduction}\label{sec:introduction}
 Let $\Omega\subset \R^2$ be a bounded open set and 
$v = (v_1,v_2):\Omega\rightarrow\R^2$ be a 
map of class $ C^1(\Omega;\R^2)$.
The area functional 
$\mathcal A(v;\Omega)$ computes the $2$-dimensional Hausdorff measure 
$\mathcal H^2$
of the graph 
\begin{align}\label{graph_u}
G_v:=\{(x,y)\in \Omega\times \R^2:y=v(x)\}
\end{align}
of $v$, 
a Cartesian $2$-manifold in $\Omega\times\R^2\subset\R^{4}$, 
and is given by the classical formula\footnote{Clearly,
\eqref{classical_area} is  finite if $v\in W^{1,1}(\Om;\R^2)$ and $\det\nabla v\in L^1(\Om)$.}
\begin{align}\label{classical_area}
\mathcal A(v;\Omega)=
\int_\Omega\sqrt{1+|\nabla v_1|^2+|\nabla v_2|^2+(\det \nabla v)^2}dx,
\end{align}
where 
\begin{align}
\det \nabla v
=\frac{\partial v_1}{\partial x_1}\frac{\partial v_2}{\partial x_2}-\frac{\partial v_1}{\partial x_2}\frac{\partial v_2}{\partial x_1}
\end{align}
 is the Jacobian determinant of $v$.
As opposite to the case when the map is scalar-valued, the functional $\mathcal A(\cdot, \Omega)$ 
is not convex, but only 
polyconvex in $\nabla v$, and its growth 
is not linear, due to the presence of ${\rm det}(\grad v)$. 

An interesting problem is to try
to extend $\mathcal{A}(\cdot\,;\Omega)$ 
out of $C^1(\Om; \R^2)$: setting for convenience
$$
\mathcal{A}(v;\Omega):
=+\infty \qquad \forall v\in L^1(\Om; \R^2)
\setminus C^1(\Omega;\R^2),$$
let us consider the sequential 
lower semicontinuous envelope
\begin{equation}\label{relaxed}
\overline{\mathcal{A}}_\tau(\mappa;\Omega)
:=\inf\left\{\liminf_{k\rightarrow 
+\infty}\mathcal{A}(v_k;\Omega):(v_k)\subset C^1(\Omega;\R^2)\cap S,\;v_k\xrightarrow{\tau} \mappa\right\}
\qquad
\forall \mappa\in S
\end{equation}
of $\mathcal{A}(\cdot\,;\Omega)$
with respect to a metrizable topology 
$\tau$ on a subspace $S\subseteq L^1(\Om; \R^2)$ containing 
those $v \in C^1(\Om; \R^2)$ with $\mathcal A(v; \Om)<+\infty$, 
and choose this 
as the extended notion of area. 

A typical choice is $S = L^1(\Om; \R^2)$ and 
 $\tau$ 
the $L^1(\Omega;\R^2)$ topology, i.e.,
$\overline{\mathcal A}_{\tau} = 
\overline{\mathcal{A}}_{L^1}$, a
case in which little is 
known\footnote{
For scalar valued maps
it is known that the domain of 
$\overline{\mathcal{A}}_{L^1}(\cdot; \Omega)$
 is $BV(\Omega)$, and on $BV(\Omega)$ 
the relaxed functional can be represented as the right-hand side 
of \eqref{eq:A_bar_Luno}, 
see \cite{DalMaso:80,Giusti:84}.}.
It is not difficult to show that the domain of 
$\overline{\mathcal{A}}_{L^1}$
is properly contained
in $BV(\Omega;\R^2)$, but its characterization is not available. 
Also, one can prove that
\begin{equation}\label{eq:A_bar_Luno}
\overline{\mathcal{A}}_{L^1}(\mappa;\Omega)\geq \int_\Omega\sqrt{1+|\nabla \mappa|^2}dx+
|D^s \mappa|(\Omega), 
\end{equation}
but the inequality might be strict \cite{AD}. 
Here $\nabla \mappa$ is the approximate gradient of $\mappa$, 
$\vert \cdot \vert$ is the Frobenius norm, 
$D^s \mappa$ is the singular part of the distributional gradient $D\mappa$ 
of $\mappa$, and 
$|D^s \mappa|(\Omega)$ stands for the total variation of 
$D^s \mappa$.  
Finding
the expression of 
$\overline{\mathcal{A}}_{L^1}(\cdot\,;\Omega)$ is possible, at the moment,
only in very special cases.
This is also due to its nonlocal 
behaviour,
since for several maps $\mappa$,
the set function $U\mapsto 
\overline{\mathcal{A}}_{L^1}(\mappa;U)$ 
is not sub-additive with respect
to the open set $U\subseteq \Om$. This happens, for example,  for the symmetric triple junction map $u_T$ on an open disk $B_\ell$, as conjectured in 
\cite{DeGiorgi:92}, and proven in \cite{AD}. A complete
picture can be found in 
\cite{BP,S}, where  
$\overline{\mathcal{A}}_{L^1}(u_T;B_\radius)$ is explicitely computed,
taking advantage of the symmetry of the map and of $B_\radius$. 
We refer also to \cite{BEPS} where an upper bound inequality is proved for a 
triple junction map without symmetry assumptions.

Also for 
the vortex map $u_V:B_\radius\setminus\{0\}\rightarrow\Suno$, 
\begin{equation}\label{vortex}
u_V(x):=\frac{x}{|x|},
\end{equation}  
the above mentioned nonsubadditivity holds. 
In \cite{AD} it is proved that 
\begin{align}\label{area_vortice_L1}
\overline{\mathcal{A}}_{L^1}(u_V;B_\radius)=\int_{B_\radius}\sqrt{1+|\nabla u_V|^2}dx+\pi \qquad
{\rm if}~ 
\radius {\rm ~ is~ sufficiently~ large},
\end{align}
while 
\begin{align}\label{area_vortice_L1_small}
\overline{\mathcal{A}}_{L^1}(u_V;B_\radius)
<\int_{B_\radius}\sqrt{1+|\nabla u_V|^2}dx+\pi \qquad
{\rm if}~ 
\radius {\rm ~is~  sufficiently~ small}.
\end{align}

The explicit computation
of $\overline{\mathcal{A}}_{L^1}(u_V;B_\radius)$
for small values of $\radius$ has been done in 
 \cite{BES}, again strongly exploiting the symmetries, where it is shown that 
$\overline{\mathcal{A}}_{L^1}(u_V;B_\radius)$
is related to a Plateau-type problem in codimension $1$, 
whose solution is a sort of (half) 
catenoid constrained to contain a segment. 
This ``catenoid'' 
describes the  
vertical part of a Cartesian current 
obtained as a limit of the graphs 
of a recovery sequence. 
Specifically, 
the main result in \cite{BES} reads as 
 \begin{align}\label{value_vortex}
\overline{\mathcal{A}}_{L^1}(u_V;B_\radius)
 =\int_{B_\radius}\sqrt{1+|\nabla u_V|^2}dx+\inf \mathcal F_\varphi(h,\psi),
 \end{align}
 where the infimum is taken over all functions $h\in C^0([0,2\radius];[-1,1])$ with $h(0)=h(2\radius)=1$, and $\psi\in BV((0,2\radius)\times (-1,1))$ with $\psi=0$ on $UG_h$, 
and
\begin{equation}
\label{area_vortice_generale}
 \begin{aligned}
 \mathcal F_\varphi(h,\psi)
= &  
\int_{(0,2\radius)\times (-1,1)}
\sqrt{1 + \vert \grad \psi\vert^2}~dtds 
+ \vert D \psi\vert((0,2\radius)\times (-1,1)
)
\\
& +\int_{((0,2\radius)\times \{-1,1\})\cup (\{0,2\radius\}\times (-1,1))}|\psi-\varphi|d\mathcal H^1-|UG_h|,
 \end{aligned}
\end{equation}
 where $\varphi:\R\times [-1,1]\rightarrow\R$ is $\varphi(t,s)=\sqrt{1-s^2}$, and $UG_h$ is the region in $[0,2\radius]\times[-1,1]$ upon the graph of $h$.
 The latter functional accounts for a Plateau problem in non-parametric form with partial free boundary on a plane domain (see also \cite{BMS} 
for more details).
 If $\radius$ is 
large enough, a minimizer of $\mathcal F_\varphi$ has the shape of two half-disks of radius $1$, whose total area is $\pi$, recovering the result in \eqref{area_vortice_L1}.
 
 \medskip
 
The $L^1$-topology is rather weak, 
and so it is convenient in order to show compactness results,
in the effort of proving existence of minimizers of some possible 
weak formulation
of the two-codimensional Cartesian Plateau problem. However,
the above discussion
 illustrates the 
difficulties of the study of the corresponding relaxation problem. 
Besides all nonlocality phenomena, 
the  $L^1$ convergence
 does not provide any control on the derivatives of $v$ and, 
of course, neither on the Jacobian determinant.
The aim of the present paper is to study the
relaxation of the area in $S = BV(\Om; \R^2)$ in a different topology, 
stronger
than the $L^1$-topology, 
in order to avoid nonlocality and keep some control of the
gradient terms.   
Specifically, we will take as $\tau$ in \eqref{relaxed} 
the topology induced by the 
strict convergence in $BV(\Omega;\R^2)$. This notion of convergence, weaker than the strong $W^{1,1}$ topology, and in general not related with the weak $W^{1,1}$ topology (see Remark \ref{rem_2.3}), allows to consider relaxation in \eqref{relaxed} for all $BV$-maps.
We recall that $(v_k)$ 
converges to $\mappa$ strictly 
$BV(\Omega;\R^2)$ if $v_k\rightarrow \mappa$ in $L^1(\Omega;\R^2)$ 
and $|Dv_k|(\Omega)\rightarrow |D\mappa|(\Omega)$ (see  Section 
\ref{sec_SC} for details). We are therefore led
to consider, for all $\mappa\in BV(\Omega;\R^2)$,
 the corresponding 
relaxed area functional $\overline{\mathcal{A}}_{\tau}=
\overline{\mathcal{A}}_{BV}$, 
\begin{equation}\label{relaxed BV}
\overline{\mathcal{A}}_{BV}
(u;\Omega):=\inf\left\{\liminf_{k\rightarrow+\infty}
\mathcal{A}(v_k;\Omega): (v_k)\subset C^1(\Omega;\R^2)\cap BV(\Omega;\R^2),{v_k\rightarrow \mappa}\mbox{ strictly } BV(\Omega;\R^2)\right\}.
\end{equation}
In the first part of the paper we restrict our analysis 
to maps $w:B_\radius \setminus \{0\}\rightarrow 
\Suno=\{x\in\R^2:|x|=1\}$ of the form
\begin{equation}\label{def_w}
w(x)=\varphi(u_V(x))=\varphi\left(\frac{x}{|x|}\right),
\end{equation}
with $\varphi:\Suno \rightarrow \Suno$ Lipschitz 
continuous.
The vortex map corresponds to the case $\varphi=\mathrm{id}$.

After setting some notation and preliminaries in Section \ref{sec:preliminaries},
in particular the total variation of the Jacobian,
the Jacobian distributional determinant ${\rm Det} \grad u$ (Section \ref{sec:the_Jacobian_determinant_and_its_total_variation}), 
and the degree (Section \ref{sec:multiplicity_and_degree}), 
in Section \ref{sec_varphi} we prove the following result:

\begin{Theorem}\label{main result 2}
Let $\radius >0$,
and $w:B_\radius \setminus \{0\}\rightarrow
\Suno$ be as in \eqref{def_w}. 
Then
	\begin{equation}
	\overline{\mathcal{A}}_{BV}(w;B_\radius)=\int_{B_\radius}\sqrt{1+|\grad w|^2}dx + \pi|\grado(\varphi)|.
	\end{equation}
In particular,
	\begin{equation}\label{main result}
	\overline{\mathcal{A}}_{BV}(\vortexmap;B_\radius)=\int_{B_\radius}\sqrt{1+|\grad u_V|^2}dx + \pi.
	\end{equation}
\end{Theorem}
By \eqref{area_vortice_L1}, for $\radius$ 
large enough we find 
$\overline{\mathcal{A}}_{BV}(\vortexmap;B_\radius)
=\overline{\mathcal{A}}_{L^1}(\vortexmap;B_\radius)$ 
while by \eqref{area_vortice_L1_small}, for small values of $\radius$ we have
$\overline{\mathcal{A}}_{BV}(\vortexmap;B_\radius)>
\overline{\mathcal{A}}_{L^1}(\vortexmap;B_\radius)$.
We also remark that 
for any radius $\radius$, 
in the computation of $\overline{\mathcal A}_{BV}(\vortexmap; B_\radius)$, 
the minimal surface employed to fill the holes of the graph $\mathcal G_{\vortexmap}\subset\R^4 $ of $\vortexmap$
 is a two dimensional disc living upon the origin of $\R^2$.

In Section \ref{sec_general} we extend our analysis to 
a more general class of maps $u\in W^{1,1}(\Omega;\Suno)$. To state our result, we recall that when
$\vert \textrm{Det}\nabla u\vert (\Om)<+\infty$, then 
$\textrm{Det}\nabla u$ 
can be written as 
$$\textrm{Det}\nabla u=\pi\sum_{i=1}^m d_i\delta_{x_i},$$
where the points $x_i \in \Om$ are the 
topological singularities of $u$, 
around which the degree of $u$ is nontrivial and equals $d_i\in \mathbb Z\setminus \{0\}$
(see Theorem \ref{Brezis thm}).
We then prove the following:

\begin{Theorem}\label{relaxed area BV}
	Let $u\in W^{1,1}(\Omega;\Suno)$. Suppose that 
$\textrm{Det}\nabla u$ 
is a Radon measure with
finite total variation
$\vert \textrm{Det}\nabla u\vert(\Om)$.  Then
	\begin{equation}\label{BV result}
	\overline{\mathcal{A}}_{BV}(u;\Omega)=
\int_{\Omega}\sqrt{1+|\grad u|^2}dx
+ \vert {\rm Det} \grad u\vert(\Om) =
 \int_{\Omega}\sqrt{1+|\grad u|^2}dx+
\pi\sum_{i=1}^N|d_i|,
	\end{equation}
	where $N\in\N$ and $d_1,\dots, d_N\in \mathbb Z\setminus \{0\}$ are such that $\textrm{Det}\nabla u=\pi\sum_{i=1}^N d_i\delta_{x_i}$.
\end{Theorem}

The total variation of $\textrm{Det}\nabla u$ can be 
characterized by relaxation.
More precisely, for maps $u \in W^{1,2}_{\textrm{loc}}(\Om; \R^2)$,
 we introduce the functional
$\totvarjac(\mappabuona;\Om):=\int_\Om|{\rm det} \grad \mappabuona|dx$, 
 measuring the total 
variation of the Jacobian of $v$, 
and consider
$$\totvarjac
_{W^{1,1}}(u;\Om):=\inf\left\{\liminf_{k\rightarrow+\infty}
\totvarjac
(\mappabuona_k;\Om):
(\mappabuona_k)\subset C^1(\Om;\R^2)\cap W^{1,1}(\Om;\R^2),\; 
\mappabuona_k\rightarrow \mappa \text{ in }W^{1,1}(\Om;\R^2)\right\},$$
for all $u\in W^{1,1}(\Om;\R^2)$. It is known (see Theorem \ref{Brezis thm})
that for 
$u$ as in Theorem \ref{relaxed area BV}, 
$$\totvarjac_{W^{1,1}}(u;\Om)=|\textrm{Det}\nabla u|(\Om).
$$
In Theorem \ref{teo:relaxation_of_TV_in_the_strict_convergence}
we show that 
$$\totvarjac_{W^{1,1}}(u;\Om)=\totvarjac_{BV}(u;\Om),$$
where
$$ 
\totvarjac_{BV}(u;\Om):=\inf\left\{\liminf_{k\rightarrow+\infty}
\totvarjac
(\mappabuona_k;\Om):
(\mappabuona_k)\subset C^1(\Om;\R^2)\cap BV(\Om;\R^2),\; \mappabuona_k\rightarrow u \text{ strictly }BV(\Om;\R^2)\right\}.
$$
Eventually, in Section \ref{sec_triple pt} we 
consider some piecewise constant maps valued in $\Suno$, 
in particular the symmetric triple junction map (see Section  \ref{sec_triple pt} for the precise definition). 
If we  call $T_{\alpha\beta\gamma}$ the equilateral 
triangle with vertices $\alpha,\beta,\gamma \in \Suno$ and $L
:= \vert \beta-\alpha\vert$ its side length, then we have:
\begin{Theorem}\label{teo:symmetric_triple-point_map}
Let $\triplemap:
B_{\radius}:=B_{\radius}(0)
\rightarrow \{\alpha,\beta,\gamma\}$ be the symmetric 
triple-point map. Then 
\begin{equation}\label{eq:relaxed_triple-point}
\overline{\mathcal{A}}_{BV}(\triplemap;B_\radius)
=\vert B_\radius\vert+ L \mathcal H^1(J_{\triplemap}) + 
\vert T_{\alpha \beta\gamma}\vert,
\end{equation}
where $\vert \cdot \vert$ is the Lebesgue measure and $J_{\triplemap}$ is 
the jump set  
of $\triplemap$.
\end{Theorem}
In particular, in view of the 
results in \cite{AD}, \cite{BP}, we find
$\overline{\mathcal{A}}_{BV}(\triplemap;B_\radius) > 
\overline{\mathcal{A}}_{L^1}(\triplemap;B_\radius)$. 
We will also see that the same argument 
used to prove Theorem \ref{teo:symmetric_triple-point_map}
provides a proof also for a symmetric $n$-uple junction function, as 
expressed in Corollary \ref{n point area}.
 
As opposite to $\overline {\mathcal A}_{L^1}(u; \Om)$, 
we see that 
the functional $\overline{\mathcal A}_{BV}(u; \Om)$, 
at least for the maps $u$ taking values in $\Suno$
considered here, is local, and admits an 
integral representation. 
 
We conclude this introduction by pointing out that, at the 
present stage, we miss the generalization of our results 
in higher dimension or codimension. On the one hand the strict  convergence in $BV$ provides some control on the gradient of $u$, and consequently, on the distributional determinant. 
In the case of maps $u:\Om\subset\R^3\rightarrow\R^3$,
for instance,
this notion of convergence might be useful to get some control of the $2\times2$-subdeterminants of $\nabla u$, but seems too weak to control  the higher order minors. On the other hand, even in the case of maps $u:\Om\subset\R^3\rightarrow\R^2$, the strict  convergence in $BV$ is not 
sufficient to show the counterpart of Proposition \ref{strict implies uniform} (see Remark \ref{rem1}) which, in our arguments, 
is crucial to localize the concentrations of $|\det\nabla v_k|$ (where $(v_k)$ is
a sequence converging to $u$).

\section{Preliminaries}\label{sec:preliminaries}
In this section we collect some preliminaries. For an integer $M \geq 2$,
set $\mathbb S^{M-1} := \{x \in \R^M : \vert x\vert =1\}$.

\begin{Theorem}[\textbf{Reshetnyak}]\label{Reshetnyak}
	Let $\Omega\subseteq \Rn$ be an open set and $\mu_h,\mu$ 
be (finite) Radon measures valued in $\R^M$. Suppose that $\mu_h\stackrel{*}{\rightharpoonup}\mu$ and $|\mu_h|(\Omega)\rightarrow|\mu|(\Omega)$. Then 
	$$
	\lim_{h\rightarrow+\infty}\int_{\Omega}f\left(x,\frac{\mu_h}{|\mu_h|}(x)\right)d|\mu_h|(x)=\int_{\Omega}f\left(x,\frac{\mu}{|\mu|}(x)\right)d|\mu|(x)
	$$
	for any continuous bounded function 
$f:\Omega\times\mathbb S^{M-1}\rightarrow\R$.
\end{Theorem}

\begin{proof}
See for instance
\cite[Theorem 2.39]{AFP}.
\end{proof}

\noindent

\subsection{Strict $BV$-convergence}\label{sec_SC}
In what follows, $\Om \subset \R^2$ is a bounded open set.
For any $u\in BV(\Omega;\R^2)$, 
the distributional derivative $Du$ is a Radon measure valued in $\R^{2\times2}$. The symbol $|Du|(\Omega)$ stands for the total variation of $Du$ (see \cite[Definition 3.4, pag. 119]{AFP},  
with $\vert \cdot \vert$ the Frobenius norm).

\begin{Definition}[\textbf{Strict convergence}]
	Let $\mappa\in BV(\Omega;\R^2)$ and $(\mappa_k)
\subset BV(\Omega;\R^2)$. 
We say that $(\mappa_k)$ converges
 to $\mappa$ strictly $BV$, 
if
	$$
	\mappa_k\xrightarrow{L^1}\mappa\qquad\mbox{and}\qquad|D\mappa_k|
(\Omega)\rightarrow|D\mappa|(\Omega).
	$$
\end{Definition}
The topology of the strict convergence in $BV$ is metrized by the distance 
$$
(u,v)\to \|u-v\|_{L^1(\Omega;\R^2)}+\left||Du|(\Omega)-|Dv|(\Omega)\right|, 
\qquad u,v \in BV(\Omega; \R^2).
$$ 

\begin{Remark}[\textbf{Weak convergences and strict convergence}]\label{rem_2.3}\rm
If $u_k\rightarrow u$ strictly $BV(\Om)$ then $u_k\rightharpoonup u$ 
$w^*$-$BV(\Om)$, where $u_k\rightharpoonup u$
$w^*$-$BV(\Om)$ means:
	$$u_k\xrightarrow{L^1} u\qquad\mbox{ and}\qquad \int_{\Omega}
\varphi\cdot D u_k\rightarrow\int_{\Omega} \varphi\cdot D u\qquad\forall\varphi\in C^0_c(\Omega;\R^2),$$
with $\cdot$ the scalar product in $\R^2$. A similar definition holds for vector valued maps.
	The converse is not true, already in one dimension:  
consider the sequence $(f_k)\subset W^{1,1}((0,2\pi))$,
	$$
	f_k(x):=\frac{1}{k}\sin(kx)\qquad \forall x\in (0,2\pi).
	$$
	Then $f_k\rightharpoonup 0$ weakly 
in $W^{1,1}((0,2\pi))$, so in particular $w^*$-$BV$, 
but the convergence is not strict in $BV$, since 
$\|f'_k\|_{L^1((0,2\pi))}=4$ for all $k\in\N$. 
	We underline that on $W^{1,1}(\Omega;\R^2)$ the strict $BV$ convergence
	 is not comparable with the weak convergence: 
the following  
 slight modification of  \cite[Example 4, pag. 42]{GMS}, 
 provides a sequence converging strictly $BV((0,1))$ but not weakly in $W^{1,1}((0,1))$.
	Consider the sequence $(g_k) \subset L^1((0,1))$ defined by 
	$$
	g_k(x):=2^k\sum_{i=0}^{k-1}
\chi_{\left[\frac{i}{k},\frac{i}{k}+\frac{1}{k2^k}\right]}(x)\qquad\forall x\in[0,1],  \ \forall k \geq 1,
	$$
where $\chi_A$ is the characteristic function of the set $A$. 
Then $\|g_k\|_{L^1}=1$ for every $k\in\N$. 
Now, let $f_k\in C([0,1])$ be the primitive of $g_k$ vanishing at $0$; then 
$(f_k)$ converges uniformly to the identity, and
	$\|f_k'\|_{L^1}=\|g_k\|_{L^1}=1=\|\mathrm{id}'\|_{L^1}$
for any $k\in\N$, and
	so $f_k\rightarrow \mathrm{id}$ strictly $BV((0,1))$.
	On the other hand, $(f'_k)$ cannot converge weakly in $L^1$ since it is not equi-integrable (see \cite[Theorem 2, pag. 50]{GMS}),
since $g_k$ tends to concentrate a large mass in arbitrarily small
sets, as $k$ becomes large.
\end{Remark}

However, 
the following result (needed in the proof of Propositions \ref{lower bound2}
and \ref{lower bound TVBV})
 shows that the strict $BV$ convergence implies the uniform one, under certain hypotheses.

\begin{Proposition}[\textbf{Strict convergence in one dimension}]
\label{strict implies uniform}
	Let $I=(a,b)\subset\R$ be a bounded interval and let $(f_k)$ be a sequence in $W^{1,1}((a,b))$. 
Suppose that $(f_k)$ converges strictly $BV((a,b))$ to $f\in W^{1,1}((a,b))$. Then $f_k\rightarrow f$ uniformly in $(a,b)$.
\end{Proposition}


\begin{proof}
	First of all, for any open interval $J\subset I$ we have
	\begin{equation}\label{convergence on J}
	\lim_{k\rightarrow+\infty}\int_{J}|f_k'|dx=\int_{J}|f'|dx.
	\end{equation}
	Indeed, since $f_k\rightharpoonup f$ $w^*$-$BV(I)$, by 
the lower semicontinuity of the variation, one has
	$$
	\int_J|f'|dx\leq\liminf_{k\rightarrow+\infty}\int_{J}|f_k'|dx.
	$$
	On the other hand, using the strict $BV$ convergence on $I$ and 
again the lower semicontinuity of the variation, we get
	\begin{equation*}
	\begin{aligned}
& \int_{J}|f'|dx
=
\int_{\overline{J}}|f'|dx=\int_{I}|f'|dx
-
\int_{I\setminus\overline{J}}|f'|dx
\geq
\lim_{k\rightarrow+\infty}\int_{I}|f_k'|dx
-\liminf_{k\rightarrow+\infty}\int_{I\setminus\overline{J}}|f_k'|dx
\\
= & 
\limsup_{k\rightarrow+\infty}\left(\int_{I}|f_k'|dx-\int_{I\setminus\overline{J}}|f_k'|dx\right)
=\limsup_{k\rightarrow+\infty}\int_{J}|f_k'|dx,
	\end{aligned}
	\end{equation*}
	so \eqref{convergence on J} holds.\\
	Now, since $f$ and $f_k$ belong to $W^{1,1}(I)$, we may
assume that they are continuous. By contradiction, suppose that $(f_k)$ 
does not converge uniformly to $f$, so that, 
up to a not relabeled subsequence, 
 we may suppose:
	\begin{equation}\label{1}
	\exists\delta>0\quad\exists(x_k)\subset I\quad\exists k_0\in\N: \quad|f_k(x_k)-f(x_k)|>\delta\quad\forall k\geq k_0,
	\end{equation}
and that 
there exists $\overline{x}\in\overline{I}$ such that $x_k\rightarrow \overline x$. Now consider an open interval $E\subset\overline{I}$ such that $\overline{x}\in E$ and 
	\begin{equation}\label{2}
	\int_E|f'|dx<\frac{\delta}{4}
	\end{equation}
	(in case $\overline{x}=a$ or $\overline{x}=b$, $E$ 
is a semi-open interval). 
Using \eqref{convergence on J}, we can find an index 
$k_1\in\N$ such that $k_1\geq k_0$ and
	\begin{equation}\label{3}
	\int_E|f'_k|dx<\frac{\delta}{2} \qquad\forall k\geq k_1.
	\end{equation}
	Moreover, there exists $k_2\in\N$, $k_2\geq k_1$, 
such that $x_k\in E$ for every $k\geq k_2$. 
	Pick a point $y\in E$; then for every $k\geq k_2$, using \eqref{1}, \eqref{2}, and \eqref{3}, we have
	\begin{equation}\nonumber
	\begin{aligned}
	|f_k(y)-f(y)|&\geq-|f_k(y)-f_k(x_k)|+|f_k(x_k)-f(x_k)|-|f(x_k)-f(y)|\\
	&\geq-\int_{x_k}^y|f'_k|dx+\delta-\int_{x_k}^y|f'|dx
	\geq-\int_E|f'_k|dx+\delta-\int_E|f'|dx\\
	&\geq -\frac{\delta}{2}+\delta-\frac{\delta}{4}=\frac{\delta}{4}.
	\end{aligned}
	\end{equation}
	Hence, $(f_k)$ (as any subsequence of it) does not converge to $f$ 
pointwise at every point of $E$ which leads to a contradiction,
since $|E|>0$ and $f_k\xrightarrow{L^1(E)} f$.
\end{proof}

\begin{Remark}\label{rem1}\rm
Proposition \ref{strict implies uniform} is still valid
with the same proof
 when $f_k$ and $f$ are vector valued.
On the contrary, it is crucial that the domain is unidimensional, since counterexamples can be done already in dimension 2: 
for instance, the sequence $(f_k)$ given by 
	$f_k(x):=\max\{(1-k|x|), 0\}$, $x\in\R^2$,
	converges to $0$ 
in $W^{1,1}(\R^2)$ but not uniformly in any neighborhood of the origin.
\end{Remark}

\subsection{The Jacobian determinant and its total variation}
\label{sec:the_Jacobian_determinant_and_its_total_variation}
\begin{Definition}[\textbf{Total variation of the Jacobian}]
\label{def:total_variation_of_the_Jacobian}
Let $u\in W^{1,2}_{{\rm loc}}(\Omega;\R^2)$. We define the total variation of the Jacobian of $u$  as
\begin{equation}\label{eq:TV}
	\totvarjac(u;\Omega)=\int_\Omega|\mathrm{det}\grad u|dx.
\end{equation}
\end{Definition}
We need 
to define $\totvarjac(\cdot; \Om)$ for other Sobolev maps, 
in particular for maps with singularities,  the main example 
being the vortex map $\vortexmap$ in \eqref{vortex}.
This can  be accomplished in two ways. The first one is to define the
 distributional 
Jacobian determinant Det$\grad u$: if\footnote{If 
$p=2$ then $u\in W^{1,2}(\Omega;\R^2)$.
}
 $p \in [1,2)$ and 
$u\in W^{1,p}(\Omega;\R^2)
\cap L^\infty_{{\rm loc}}(\Omega;\R^2)$, 
\begin{equation}\label{distributional det}
<\mathrm{Det}\grad u,\varphi>:=-\frac{1}{2}\int_\Omega\mathrm{adj}\grad u(x)u(x)\cdot \nabla\varphi(x)dx\qquad\forall\varphi\in{C}^\infty_c(\Omega),
\end{equation}
where ${\rm adj}\grad u := 
\left(
\begin{array}{cc}
\frac{\partial u_2}{\partial y} 
& -\frac{\partial u_1}{\partial y} 
\\
- \frac{\partial u_2}{\partial x}  
& \frac{\partial u_1}{\partial x} 
\end{array}
\right)$.
This definition is 
justified by 
the property 
$$
u \in C^2(\Om; \R^2) \Rightarrow \mathrm{det}\grad u=\frac{1}{2}\mathrm{div}(\mathrm{adj}\grad u\, u).
$$
Notice that, if $u\in C^2(\Om; \R^2)$ and $B_r(x)\subset\subset \Om$, 
then by the divergence theorem, writing the outward
unit normal to $\partial B_r(x)$  as
$\nu = (\nu_1,\nu_2)$, and its $\pi/2$-counterclockwise
rotation $\nu^\perp= \tau = (\tau_1,\tau_2)$, 
\begin{equation}\label{int_det_bordo}
\begin{aligned}
\int_{B_r(x)}\mathrm{det}\grad u\,dz
=&
\frac12
\int_{\partial B_r(x)}(\mathrm{adj}\grad u\, u)\cdot \nu~ d\mathcal H^1
\\
=& 
\frac{1}{2}
\int_{\partial B_r(x)}
\left(
\big(
\frac{\partial u_2}{\partial y} u_1
-
\frac{\partial u_1}{\partial y} u_2
\big)\nu_1 + 
\big(
-\frac{\partial u_2}{\partial x} u_1
+
\frac{\partial u_1}{\partial x} u_2
\big)
\nu_2\right) d\mathcal H^1
\\
=& 
\frac{1}{2}
\int_{\partial B_r(x)}
\left(
u_1
\big(
\frac{\partial u_2}{\partial y},  
-\frac{\partial u_2}{\partial x}\big) \cdot \nu +  
u_2
\big(-\frac{\partial u_1}{\partial y},  
\frac{\partial u_1}{\partial x}\big) \cdot \nu 
\right) 
d\mathcal H^1
\\
=&
\frac{1}{2}
\int_{\partial B_r(x)}
\left(
u_1 \grad u_2 \cdot \tau -u_2 \grad u_1 \cdot \tau
\right) 
d\mathcal H^1
\\
=& \frac{1}{2}\int_{\partial B_r(x)}\Big(u_1\frac{\partial u_2}{\partial s}-u_2\frac{\partial u_1}{\partial s}\Big)ds.
\end{aligned}
\end{equation} 
By \cite[Formula (3.7)]{Mu} (which in turn is a consequence of Theorem 3.2 in \cite{Mu}), one sees that formula \eqref{int_det_bordo} is valid also for $u\in W^{1,\infty}(\Om;\R^2)$.
	
	We recall that  
	$$
	\mathrm{Det}\grad u=\mathrm{det}\grad u\quad\forall u\in W^{1,2}(\Omega;\R^2), 
	$$ 
	while if $p \in [1,2)$ they can differ,
for instance ${\rm det}\grad \vortexmap$ is
null, whereas ${\rm Det}\grad \vortexmap=\pi\delta_0$ (see \cite{Pa}).
Then one is led to define $\totvarjac(u; \Om)=\vert {\rm Det} \grad u\vert(\Om)$,
for those $u$ for which ${\rm Det} \grad u$ is a Radon measure
with finite total variation in $\Om$.

The second way is to argue by relaxation.
For $p\in[1,2]$
 and $u\in W^{1,p}(\Omega;\R^2)$ one sets
\begin{equation}\label{relaxed TV}
\begin{split}
&
\totvarjac_{W^{1,p}}(u;\Omega)
:=\inf\left\{\liminf_{k\rightarrow+\infty}\totvarjac(\mappabuona_k;\Omega): 
(\mappabuona_k)\subset C^1(\Omega;\R^2)\cap W^{1,p}(\Omega;\R^2),
\mappabuona_k\rightarrow u\mbox{ in }W^{1,p} \right\}.
\end{split}
\end{equation}
It is known that $\totvarjac(u; \Omega) = \totvarjac_{W^{1,2}}(u;\Omega)$
for $u \in W^{1,2}(\Om; \R^2)$. Moreover, 
 when $p\in [1,2)$, 
$\totvarjac_{W^{1,p}}(\cdot;\Omega)$ coincides
with the total variation of the Jacobian distributional determinant of $u$, 
provided $u\in  W^{1,p}(\Om;\Suno)$
(see Theorem \ref{Brezis thm} below, and \cite[Theorem 11 and Remark 12]{BMP}). The same conclusions do not 
hold in general, for maps in $W^{1,p}(\Om;\R^2)$ which do 
not take values in $\Suno$ (see \cite[Open problem 5]{BMP}).
Notice also that relaxation in \eqref{relaxed TV} can also be done with respect to the weak convergence in $W^{1,p}$ (we do not treat this in the present paper and refer the reader to \cite{BMP,DP,Pa}).

We emphasize that we required $C^1$-regularity for the approximating sequences in \eqref{relaxed TV}. 
This ensures that such sequences are contained in $W^{1,2}_{\text{loc}}(\Om; \R^2)$ 
which is the minimal feature to guarantee that ${\rm det} \grad v_k
\in L^1_{\text{loc}}(\Om)$. Replacing the $C^1$-regularity 
with the $W^{1,2}_{\text{loc}}$-regularity\footnote{As sometimes can be found in literature.} gives rise to the same relaxed functionals; 
this can be seen by a density argument, since any $v\in W^{1,2}_{\text{loc}}(\Om;\R^2)$ can be approximated by maps $\mappabuona_k \in C^1(\Om; \R^2)$ 
in $W^{1,2}_{\text{loc}}(\Om; \R^2)$ (such a
convergence ensures the corresponding convergence of $\totvarjac(\mappabuona_k;\Om)$ 
to $\totvarjac(v;\Om)$). 
	In the same way, one can also replace the $C^1$-regularity with the $C^\infty$-regularity.

One can also relax $\totvarjac$ with respect to the strict $BV$ convergence:
this will be the content of Theorem \ref{teo:relaxation_of_TV_in_the_strict_convergence}. Moreover, the relaxation 
with respect to the $L^1$ convergence is possible, but not interesting for us, because 
we will deal with maps with values in $\Suno$, so the resulting relaxed functional turns out to be zero (see \cite[Corollary 5]{BMP}).

\subsection{Multiplicity and degree}\label{sec:multiplicity_and_degree}
In what follows 
$\sourceballrx$ denotes the open ball of $\R^2$ centered at $x$ of radius $r>0$.
\begin{Definition}[\textbf{Multiplicity}]
Given $u \in W^{1,1}(\Omega;\R^2)$, 
for all measurable sets $A\subseteq\Om$ and all $y\in \R^2$,  we set
$$
\mathrm{mult}(u,A,y) :=\sharp\{u^{-1}(y)\cap A\cap \mathcal R_u\},
$$
where $\mathcal R_u\subseteq \Om$ 
is the set of regular points of $u$ (see \cite[pag. 202]{GMS}). 
Similarly, if 
$u \in W^{1,1}(\partial B_r(x);\Suno)$, we define 
$$
\mathrm{mult}(u,A,y) :=\sharp\{u^{-1}(y)
\cap A\cap \mathcal R_u\},
$$
for all measurable sets $A\subseteq\partial B_r(x)$ and all $y\in \Suno$. 
\end{Definition}

 Let $u \in W^{1,1}(\Omega;\R^2)$;  by \cite[Theorem 1-6, Section 3.1.5]{GMS}, if ${\rm det} \grad u  \in L^1(\Om)$, we have 
\begin{equation}\label{eq:det_mult}
           \int_A|\det\nabla u|dx=\int_{\R^2}
\mathrm{mult}(u,A,y)
dy, 
\end{equation}
for any measurable set $A\subseteq\Om$.
In particular, 
$\mathrm{mult}(u,A,\cdot)$ is measurable and 
finite a.e. in  $\R^2$.

If a Lipschitz continuous map $\varphi:\partial B_r(x)\rightarrow\Suno$ has constant multiplicity on $\partial B_r(x)$, 
then we will make use of the simplified notation 
$$ \mathrm{mult}(\varphi):=\mathrm{mult}(\varphi,\partial B_r(x),\cdot).$$

\begin{Definition}[\textbf{Degree}]
Given $u \in W^{1,1}(\Omega;\R^2)$
with ${\rm det} \grad u \in L^1(\Om)$, 
for all measurable sets $A\subseteq\Om$,  we let
\begin{equation}\label{eq:deg_mult}
\grado(u,A,y) :=\sum_{x\in u^{-1}(y)\cap A\cap \mathcal R_u}{\rm sign}(\det\nabla u(x)),\\
\end{equation}
for those $y\in \R^2$ for which 
$\mathrm{mult}(u,A,\cdot)$ is 
finite.
\end{Definition}
Clearly
\begin{equation}\label{eq:mult_vert_deg}
\mathrm{mult}(u,A,\cdot)\geq|\grado(u,A,\cdot)|.
\end{equation}
 By \cite[Theorem 1-6, Section 3.1.5]{GMS}, if ${\rm det} \grad u \in L^1(\Om)$, then 
\begin{equation}\label{degreeintegral}
	\int_A\det\nabla u\,dx=\int_{\R^2}\grado(u,A,y)dy,
\end{equation}
for any measurable set $A\subseteq\Om$, and by \eqref{eq:det_mult} and
\eqref{eq:mult_vert_deg}
\begin{align}\label{213}
\int_\Om|\det\nabla u|dx\geq\int_{\R^2}|\grado(u,\Om,y)|dy.
\end{align}
\begin{Remark}\label{rem:too_weak}\rm
The notion \eqref{eq:deg_mult} of degree is too weak to be related to the trace of $u$ on $\partial \Om$. However, homological invariance is 
recovered under stronger hypotheses on $u$; 
for instance  if $u,v$ are Lipschitz in 
$\widehat \Om\supset\supset\Om$ and $u=v$ in $\widehat\Om
\setminus \overline\Om$, then 
$\grado(u,\Om,\cdot)=\grado(v,\Om,\cdot)$ a.e. 
in $\R^2$ (see \cite[pag. 233 and 469]{GMS}). 
In particular, if $u,v:\sourceballrx\rightarrow \R^2$ 
are Lipschitz continuous and $u=v$ on $\partial \sourceballrx$, 
then we might extend $u$ to a Lipschitz map $\overline u$ on $\R^2$; 
the map $\overline v$ coinciding with $v$ in $\sourceballrx$ 
and with $\overline u$ outside $\sourceballrx$ 
is  a Lipschitz extension of $v$. 
Hence $\grado(\overline u,\sourceballrx,\cdot)
=\grado(\overline v,\sourceballrx,\cdot)$, 
which implies $\grado(u,\sourceballrx,\cdot)=\grado(v,\sourceballrx,\cdot)$.
\end{Remark}

\begin{Definition}
For an open disc $\sourceballrx\subset\R^2$ and $u\in W^{1,1}(\partial B_r(x);\Suno)$, we define 
\begin{equation}\label{eq:grado}
\grado(u):=\frac{1}{2\pi}\int_{\partial \sourceballrx}\Big(u_1\frac{\partial u_2}{\partial s}-u_2\frac{\partial u_1}{\partial s}\Big)ds\;\in\;\mathbb Z. 
\end{equation}
If $u\in W^{1,1}(\Om;\Suno)$, $B_r(x)\subset\subset\Om$, and $u \mres 
\partial \sourceballrx \in W^{1,1}(\partial \sourceballrx; \Suno)$ (which is true
for almost every $r$), 
we set
\begin{equation}\label{eq:deg_restriction}
	\grado(u,\partial \sourceballrx)
:=\grado (u\mres \partial B_r(x)).
\end{equation}
\end{Definition}

\begin{Remark}\label{rem:coincidence_of_degrees}\rm
If $u:\sourceballrx\rightarrow\R^2$ 
is Lipschitz continuous and $|u|=1$ on $\partial \sourceballrx$, 
then $\grado(u,\sourceballrx,\cdot)$ is constant in
$B_1=B_1(0)$, and coincides with $\grado(u,\partial \sourceballrx)$. 
Indeed $\grado(u,\sourceballrx,\cdot)$ is a constant $c$ in $B_1$ 
thanks to \cite[Theorem 1.3]{Ham} (and zero on $\R^2 \setminus
B_1$), and then it is sufficient to 
check that $\grado(u,\sourceballrx,y)=\grado(u,\partial \sourceballrx)$, for a.e. $y\in B_1$. By applying  \eqref{int_det_bordo} to the left-hand side of \eqref{degreeintegral} one has 
$$
\int_{\R^2} {\rm deg}(u, \sourceballrx, y) ~dy=
\int_{B_1} {\rm deg}(u, \sourceballrx, y) ~dy=
\pi c = 
\int_{\sourceballrx} {\rm det} \grad u~dx=\pi\grado(u\mres\partial B_r(x)).
$$
In this particular case, thanks to \eqref{213}, we conclude
\begin{align}\label{stima_jacob_grado}
\int_{\sourceballrx}|\det\nabla u|dx\geq\int_{B_1}|\grado(u,\partial \sourceballrx)|dy=\pi|\grado(u,\partial \sourceballrx)|.
\end{align}
\end{Remark}

\subsection{Singular Sobolev maps with values in $\Suno$}
\label{subsec:Sobolev_maps_with_values_in_S1}
We will make use of  the following theorems.
\begin{Theorem}\label{Brezis thm}
 Let $u\in W^{1,1}(\Omega;\Suno)$. Then 
	$$
	\totvarjac_{W^{1,1}}(u;\Omega)<+\infty\Longleftrightarrow \mathrm{Det}\grad u\quad \mbox{is a Radon measure}.
	$$
	In this case $\totvarjac_{W^{1,1}}(u;\Omega)=|\mathrm{Det}\grad u|(\Om)$,
and there exists a finite set $\{x_1,\dots,x_\cardsing\}$
of points in $\Omega$ 
such that
	\begin{equation}\label{eq:Det_grad_u_sum_deltas}
	\mathrm{Det}\grad u=\pi\sum_{i=1}^\cardsing d_i\delta_{x_i},
	\end{equation}
	where $d_i=\grado(u,\partial B_{r_i}(x_i)) \in \Z\setminus \{0\}$ 
for a.e. $r_i > 0$ small enough.
	In particular 
	$$
	|\mathrm{Det}\grad u|(\Om)
=\pi\sum_{i=1}^\cardsing |d_i|.
	$$
\end{Theorem}
\begin{proof}
See for instance \cite [Theorem 11]{BMP} and \cite[Proposition 5.2]{JS}.
\end{proof}
\begin{Remark}
Theorem \ref{Brezis thm} provides the existence of a radius $r_i>0$ such that the number $d_i$ not only is the degree of the trace of $u$ on $\partial B_{r_i}(x_i)$, but also on almost every circumference $\partial B_\rho(x_i)$ with $\rho<r_i$. Moreover, on these circumferences, we may
 assume that $u$ is continuous, since its trace is still of class $W^{1,1}$. 
For more details, we refer the reader to \cite{BMP}.
\end{Remark}

\begin{Theorem}\label{teo_approx}
	Let $u\in W^{1,1}(\Suno;\Suno)$. Then there exists a sequence in
$C^\infty(\Suno;\Suno)$ converging to $u$ in $ W^{1,1}(\Suno;\Suno)$.
\end{Theorem}
\begin{proof}
See 
\cite[Theorem 2.1]{Miru}.
\end{proof}

	\begin{Theorem}\label{approssimanti lisce}
Let $B\subset \R^2$ be a bounded open connected set, and
$u\in W^{1,1}(B;\Suno)$. Then there exists a 
		sequence in 
$C^\infty(B;\Suno)$ converging
		to $u$ in $W^{1,1}(B;\Suno)$ if and only if Det$\grad u=0$ in the sense of distribution.
	\end{Theorem}
\begin{proof}
See \cite[Theorem 1.5]{PVS}.
\end{proof}

\section{Relaxation for vortex-type maps in $W^{1,p}(B_\radius;\Suno)$: Theorem \ref{main result 2}}\label{sec_varphi}
In this section we focus on maps $w\in W^{1,1}(B_\radius;\Suno)$ 
of the form \eqref{def_w}, 
where $\varphi:\Suno\rightarrow\Suno$ is a Lipschitz map. 

Of course $\det \grad w=0$ a.e. on $B_\radius$. Moreover, 
$w\in W^{1,p}(B_\radius;\Suno)$ for every $p\in[1,2)$; indeed, 
for $x\in B_\radius\setminus \{0\}$, let us write 
in polar coordinates 
\begin{equation}\label{def_f}
w(x)=\widetilde{w}(\rho,\theta)=\varphi(\cos\theta,\sin\theta)=:f(\theta)=(f_1(\theta),f_2(\theta))\qquad\forall \rho\in(0,\radius), \quad\forall\theta\in[0,2\pi).
\end{equation}
Then for a.e. $\theta\in[0,2\pi)$ and all $\rho\in(0,\radius)$
\begin{equation}\nonumber
\grad _{\rho,\theta}\widetilde 
w(\rho,\theta)=\left(\begin{matrix}
                                    0 &f'_1(\theta)\\
                                    0 &f'_2(\theta)
                                    \end{matrix}
                                    \right),
\qquad\qquad
|\grad _{\rho,\theta}\widetilde w(\rho,\theta)|=|\partial_\theta\widetilde{w}(\rho,\theta)|=|f'(\theta)|,
\end{equation}
\begin{equation}\label{variation of w}
\begin{aligned}
\int_{B_\radius}|\grad w|^pdx&=
\int_0^{2\pi}\int_0^\radius\rho\left({|\partial_\rho\widetilde{w}|^2+\frac{|\partial_\theta\widetilde{w}|^2}{\rho^2}}\right)^\frac{p}{2} d\rho d\theta
\\
&=\int_0^{2\pi}\int_0^\radius\frac{|f'(\theta)|^p}{\rho^{p-1}} d\rho d\theta\leq2\pi\mathrm{lip}(f)^p\int_0^\radius\frac{1}{\rho^{p-1}}d\rho<+\infty;
\end{aligned}
\end{equation}
in particular
\begin{equation}\label{variation of w_2}
\begin{aligned}
\int_{B_\radius}|\grad w|dx=\radius \int_0^{2\pi}{|f'(\theta)|} d\theta.
\end{aligned}
\end{equation}

\begin{Remark}
We have used that $f$ in \eqref{def_f} is 
Lipschitz continuous in $[0,2\pi)$. 
Let us  check that $\mathrm{lip}(f)=\mathrm{lip}(\varphi)$ and, 
moreover, $\mathrm{Var}(f):=\int_0^{2\pi}|f'(\theta)|d\theta=\int_{\Suno}|\nabla^{\Suno}\varphi(y)|d\mathcal H^1(y)
=\mathrm{Var}(\varphi)$, where
\begin{equation}\label{gradient of phi}
\nabla^{\Suno}\varphi(z):=\lim_{\substack{y\rightarrow z \\ y\in
\Suno \setminus\{z\}}}\frac{\varphi(y)-\varphi(z)}{|y-z|},
\end{equation}
is the (tangential) derivative of $\varphi$ 
on $\Suno$, 
that is well-defined for a.e. $z\in \Suno$ as an element of  the tangent space $T_{\varphi(z)}
\Suno$ to $\Suno$ at $\varphi(z)$. Fix  $y_0\in \Suno$ 
where $\varphi$ is differentiable, and take the unique $\theta_0\in[0,2\pi)$ such that $y_0=(\cos\theta_0,\sin\theta_0)$. From 
\eqref{gradient of phi}, it follows
\begin{equation}\label{gradient of phi equal f'}
\begin{aligned}
\nabla^{\Suno}\varphi(y_0)=\frac{d}{d\theta}_{|\theta=\theta_0}\varphi(\cos\theta,\sin\theta)=f'(\theta_0),
\end{aligned}
\end{equation}
and therefore $\mathrm{lip}(\varphi)=\mathrm{lip}(f)$.
Moreover
\begin{equation}\label{same variations}
\begin{aligned}
\mathrm{Var}(\varphi)=\int_{\Suno}|\nabla^{\Suno}\varphi(y)|d\mathcal H^1(y)=\int_0^{2\pi}|f'(\theta)|d\theta=\mathrm{Var}(f).
\end{aligned}
\end{equation}
In particular, from \eqref{variation of w_2}, we conclude
\begin{align}\label{variazione}
\int_{B_\radius} \vert \grad w\vert~dx=\radius \textrm{Var}(\varphi).
\end{align}
\end{Remark}

\begin{Remark}[\textbf{Lifting}]\label{lift}
A lifting of $\varphi$ is a map $\overline\Phi:[0,2\pi]\rightarrow\R$ such that 
\begin{equation}\label{lifting relation}
	\varphi(\cos\theta,\sin\theta)=(\cos(\overline\lift(\theta)),\sin(\overline\lift(\theta))) \qquad\forall\theta\in[0,2\pi].
\end{equation}
The function $f(\cdot)=\varphi(\cos(\cdot),\sin(\cdot)):[0,2\pi]\rightarrow \Suno$ being continuous on a simply-connected set, always admits a continuous lifting $\overline\Phi:[0,2\pi]\rightarrow\R$ such that 
$$\varphi(\cos\theta,\sin\theta)=f(\theta)=(\cos(\overline\lift(\theta)),\sin(\overline\lift(\theta))).$$ Moreover, since the covering map $t \in \R\mapsto e^{it}\in \Suno$ satisfies $|e^{it_1}-e^{it_2}|\leq |t_1-t_2|\leq \pi|e^{it_1}-e^{it_2}|$ for all $t_1,t_2$ with $|t_1-t_2|\leq\pi$, 
any continuous lifting of $\varphi$ must be Lipschitz, indeed
\begin{align}
	\frac{|\overline\Phi(\theta_1)-\overline\Phi(\theta_2)|}{|\theta_1-\theta_2|}\leq 
\pi
\frac{|e^{i\overline\Phi(\theta_1)}-e^{i\overline\Phi(\theta_2)}|}{|e^{i\theta_1}-e^{i\theta_2}|}
=
\pi
\frac{|\varphi(e^{i\theta_1})-\varphi(e^{i\theta_2})|}{|e^{i\theta_1}-e^{i\theta_2}|}
\qquad \forall \theta_1,\theta_2\in [0,2\pi] \text{ with }|\theta_1-\theta_2|\leq \pi;
\end{align}
while if $\vert \theta_1-\theta_2\vert > \pi$, the left-hand side
is bounded by $\frac{2}{\pi} \max_{[0,2\pi]} \vert \overline \Phi \vert$.

  Using the $2\pi$-periodicity of $f$, 
we see that $\overline\Phi(2\pi)-\overline\Phi(0)\in2\pi\mathbb Z$; hence $\overline\Phi$ can be extended in a Lipschitz way to the whole of $\R$ 
(this can be done extending periodically its first derivative).
It is possible to see that the lifting is unique up to a multiple of
$2\pi$: fix a starting point, e.g. $(1,0)\in\Suno$ and set
$\varphi(1,0)=:y_0\in\Suno$.
Now extract the Argument $\theta(y_0)\in[0,2\pi)$
of $y_0$, and 
define 
$\lift:\R\rightarrow\R$ as
\begin{align}\label{39}
\lift(t):=\theta(y_0)+\int_0^t\lambda_\varphi(s)ds,
\end{align}
 where $\lambda_\varphi(s)\in \R$ is uniquely determined by
\begin{align}\label{40}
\nabla^{\Suno}\varphi(\cos s,\sin s)=\lambda_\varphi(s) \tau_{\varphi(\cos s,\sin s)}\qquad \mbox{a.e. }s\in\R,
\end{align}
with 
\begin{equation}\label{41}
\tau_{\varphi(\cos s,\sin s)}=\varphi^\perp(\cos s,\sin s)=\big(-\varphi_2(\cos s,\sin s),\varphi_1(\cos s,\sin s)\big)
\end{equation}
the unit tangent vector 
to $\Suno$ (counter-clockwise oriented)
at the point ${\varphi(\cos s,\sin s)}$.
By definition, $\lift$ is Lipschitz in $\R$ since
$\mathrm{lip}(\lift)=\|\lambda_\varphi\|_{\infty}=\mathrm{lip}(\varphi)$.
In order to show the lifting property 
\eqref{lifting relation}, take a lifting $\overline\Phi:\R\rightarrow\R$ of $\varphi$. 
 Differentiating the equality
$\varphi(\cos s,\sin s)=(\cos(\overline\Phi(s)),\sin(\overline\Phi(s)))$	
gives $$\lambda_\varphi(s)\tau_{\varphi(\cos s,\sin s)}
=\overline\Phi'(s)(-\sin(\overline\Phi(s)),\cos(\overline\Phi(s)))
=\overline\Phi'(s) \tau_{\varphi(\cos s,\sin s)},\qquad \mbox{a.e. }s\in\R,$$
so that $\overline\Phi'=\lambda_\varphi$ a.e. in $\R$. This implies, by \eqref{39}, that $\Phi(t)-\overline\Phi(t)$ is a constant multiple of $2\pi$. Thus $\Phi$ also satisfies 
\eqref{lifting relation}, and any lifting of $\varphi$ is of the form \eqref{39}, up to a constant multiple of $2\pi$.

As a further consequence of the previous discussion and of \eqref{40}-\eqref{41}, for any lifting $\widetilde \Phi$ of $\varphi$, and in particular for $\Phi$, the map $\widetilde f(\theta)=(\cos(\widetilde\Phi(\theta)),\sin(\widetilde\Phi(\theta)))$ satisfies the same 
linear ordinary differential system as $f$, namely 
\begin{equation}\label{lifting system}
f'_1 =-\lift'{f_2},\qquad
f'_2 =\lift'{f_1} \qquad {\rm a.e.~in}~ \R.
\end{equation}

Finally, from \eqref{lifting system} it follows
$\lambda_\varphi=
f_1 f_2' -f_2f_1'$ a.e. in $\R$,
so that by \eqref{eq:grado}, we get
\begin{equation}\label{lifting and degree}
\begin{aligned}
\lift(2\pi)=\lift(0)+\int_0^{2\pi}\lambda_\varphi(\theta)d\theta=\lift(0)+2\pi\grado(\varphi).
\end{aligned}
\end{equation}
\end{Remark}

Now we can start the proof of Theorem \ref{main result 2}: 
As usual, we divide it into 
two parts, the lower bound
(Proposition \ref{lower bound2}) and the upper bound 
(Proposition \ref{upper bound2}).


\begin{Proposition}[\textbf{Lower bound}]\label{lower bound2}
	Let $w:B_\radius \setminus \{0\}\rightarrow
\Suno$ be the map defined in \eqref{def_w}. Suppose that $(v_k)\subset C^1(B_\radius;\R^2)\cap BV(B_\radius;\R^2)$ is such that $v_k\rightarrow w$ strictly $BV(B_\radius;\R^2)$. Then
	$$
	\liminf_{k\rightarrow+\infty}\mathcal{A}(v_k;B_\radius)\geq\int_{B_\radius}\sqrt{1+|\grad w|^2}dx + \pi|\grado(\varphi)|.
	$$
\end{Proposition}

\begin{proof}
We may assume that 
$$
\liminf_{k\rightarrow+\infty}\mathcal{A}(v_k;B_\radius)=\lim_{k\rightarrow+\infty}\mathcal{A}(v_k;B_\radius) < +\infty.
$$
We define the functions $\sliceTV_k,
\sliceTV:(0,\radius)\rightarrow[0,+\infty)$ as  
$$
\sliceTV_k(r):=\int_{\partial B_r}|\grad v_k|ds,\quad 
\sliceTV(r):=\liminf_{k\rightarrow +\infty}\sliceTV_k(r),
\qquad
r\in(0,\radius),
$$
where $s$ is an arc length parameter on $\partial B_r$. 
By Fubini's theorem it follows
$$
\int_0^\radius\sliceTV_k(r)dr=\int_{ B_\radius}|\grad v_k|dx,
$$
hence, using Fatou's lemma, the 
strict convergence of $(v_k)$ to $w$, and \eqref{variazione}, 
\begin{equation}\label{integral of phi 2}
\begin{aligned}
\int_0^\radius\sliceTV(r)dr
&\leq
\liminf_{k\rightarrow+\infty}\int_0^\radius\sliceTV_k(r)dr
=\lim_{k\rightarrow+\infty}\int_{ B_\radius}|\grad v_k|dx\\
&=\int_{B_\radius}|\grad w|dx=\radius\mathrm{Var}(\varphi).
\end{aligned}
\end{equation}
 In particular, 
$$
\sliceTV {\rm ~is~ almost~ everywhere~
finite ~in~ }  (0,\radius).
$$
Now we claim that 
\begin{equation}\label{claimw}
\sliceTV=\mathrm{Var}(\varphi)\quad \mbox{ a.e. in } (0,\radius).
\end{equation}
Indeed, without loss of generality we may 
assume that $(v_k)$ converges to $w$ almost everywhere in $B_\radius$, so 
that for almost every $r\in(0,\radius)$ 
\begin{equation}\label{a.e. convergence on circumferencesw}
v_k\mres\partial B_r\rightarrow w\mres \partial B_r\qquad \mathscr{H}^1-\mbox{a.e. in } \partial B_r.
\end{equation}
Now fix $r\in(0,\radius)$ such that \eqref{a.e. convergence on circumferencesw} holds; consider the total variation of $v_k\mres \partial B_r$, that is the $L^1(\partial B_r)$-norm of the 
tangential derivative of $v_k$ (as in \eqref{gradient of phi}):
$$
|D(v_k\mres\partial B_r)|(\partial B_r)=\int_{\partial B_r}\left|\frac{\partial v_k}{\partial s}\right|ds.
$$
Clearly
\begin{align}\label{4.14new}
\liminf_{k\rightarrow+\infty}\int_{\partial B_r}\left|\frac{\partial v_k}{\partial s}\right|ds\leq\liminf_{k\rightarrow+\infty}\int_{\partial B_r}|\grad v_k|ds=\sliceTV(r).
\end{align}
Let us 
extract a subsequence $(v_{k_h})\subset(v_k)$ depending on $r$, such that
\begin{equation}\label{sequence realizing liminfw}
\begin{aligned}
\liminf_{k\rightarrow+\infty}\int_{\partial B_r}\left|\frac{\partial v_k}{\partial s}\right|ds=\lim_{h\rightarrow+\infty}\int_{\partial B_r}\left|\frac{\partial v_{k_h}}{\partial s}\right|ds.
\end{aligned}
\end{equation}
Since $\sliceTV$ is almost everywhere finite, we may suppose that $\psi(r)<+\infty$, so that  the sequence $(v_{k_h}\mres\partial B_r)$ is bounded in $BV(\partial B_r;\R^2)$.
Thus, using \eqref{a.e. convergence on circumferencesw}, 
we also have
\begin{equation}\label{eq:4.141}
v_{k_h}\mres \partial B_r\,{\rightharpoonup}\,
w\mres \partial B_r \quad\mbox{ weakly}^* \mbox{ in } BV(\partial B_r; \R^2)\quad\mbox{as }h\rightarrow+\infty.
\end{equation} 
Now, since $\grad w$ is only tangential, and 
$\vert \grad w(r,\theta)\vert^2 = \frac{\vert f'(\theta)\vert^2}{r^2}$,
we get
\begin{equation}
\label{eq:Dw_only_tangential}
\int_{\partial B_r}\left|\frac{\partial w}{\partial s}\right|ds=\int_{\partial B_r}\left|{\grad w}\right|ds=\int_0^{2\pi}r|f'(\theta)|\frac{1}{r}d\theta=\mathrm{Var}(\varphi).
\end{equation}
Hence, using the lower semicontinuity of the variation along
$(v_{k_h}\mres \partial B_r)$, 
\eqref{sequence realizing liminfw}, and
 \eqref{4.14new} we infer
\begin{equation}\label{upper bound for phiw}
\begin{aligned}
\mathrm{Var}(\varphi)=&\int_{\partial B_r}\left|\frac{\partial w}{\partial s}\right|ds\leq\liminf_{h\rightarrow+\infty}\int_{\partial B_r}\left|\frac{\partial v_{k_h}}{\partial s}\right|ds\\
=&\lim_{h\rightarrow+\infty}\int_{\partial B_r}\left|\frac{\partial v_{k_h}}{\partial s}\right|ds=\liminf_{k\rightarrow+\infty}
\int_{\partial B_r}\left|\frac{\partial v_{k}}{\partial s}\right|ds
\leq
\sliceTV(r).
\end{aligned}
\end{equation}
Thus
  $\sliceTV\geq \mathrm{Var}(\varphi)$ almost everywhere 
in $(0,\radius)$ and, from \eqref{integral of phi 2},
we deduce $\sliceTV = \mathrm{Var}(\varphi)$ almost everywhere 
in $(0,\radius)$, and so \eqref{claimw} is proved.

As a consequence of the previous arguments, 
\begin{equation}
\label{strict convergence on circumferencesw}
\begin{aligned}
&\forall \varepsilon\in(0,\radius) \quad\exists r_\varepsilon\in(0,\varepsilon) \quad\exists(v_{k_h})\subset(v_k) \quad\mbox{s.t.}\\
& v_{k_h}\mres\partial B_{r_\varepsilon}\rightarrow w\mres\partial B_{r_\varepsilon}\quad\mbox{strictly } BV(\partial B_{r_\varepsilon};\R^2),
\end{aligned}
\end{equation}
where the subsequence $(v_{k_h})$ depends on $\eps$.
Indeed, proving \eqref{claimw}, we have shown that 
for almost every $r\in(0,\radius)$, there exists a 
subsequence $(v_{k_h})$ satisfying \eqref{eq:4.141};
so, given $\varepsilon\in(0,\radius)$,
there exists $r_\eps\in(0,\varepsilon)$ and a subsequence $(v_{k_h})$ depending on $\eps$, such that
\begin{equation}\label{conv_radius}
v_{k_h}\mres\partial B_{r_\varepsilon}{\rightharpoonup}  w\mres\partial B_{r_\varepsilon}\quad\mbox{ weakly}^* \mbox{ in } BV(\partial B_{r_\varepsilon};\R^2).
\end{equation}
But from the previous discussion
we also deduce
\begin{equation}\label{eq:lim_restriction_totvar}
\lim_{h\rightarrow+\infty}\int_{\partial B_{r_\varepsilon}}\left|\frac{\partial v_{k_h}}{\partial s}\right|ds=\sliceTV(r_\eps)=\mathrm{Var}(\varphi)=\int_{\partial B_{r_\varepsilon}}\left|\frac{\partial w}{\partial s}\right|ds;
\end{equation}
thus the convergence in \eqref{conv_radius} is actually strict in $BV(\partial B_{r_\eps}; \R^2)$.

Now, fix $\eps\in(0,\radius)$ and, for simplicity,
denote by $(v_h)$ the subsequence $(v_{k_h})$ 
for which \eqref{strict convergence on circumferencesw} holds. 
Remember that our approximating maps $v_h=((v_{h})_1,(v_{h})_2)$ are of 
class $C^1(\Om; \R^2)$, 
so they might have non-zero Jacobian determinant $J v_h := {\rm det}\grad v_h$, 
as opposed to $w=(w_1,w_2)$, whose Jacobian determinant vanishes 
a.e. in $B_\radius$. In particular, we expect the contribution of area given by $Jv_h$ to be non trivial around the origin. Thus, we 
split the area functional as follows:
$$
\mathcal{A}(v_h;B_\radius)=\mathcal{A}(v_h;B_\radius\setminus B_{r_\varepsilon})+\mathcal{A}(v_h;B_{r_\varepsilon})\geq\mathcal{A}(v_h;B_\radius\setminus B_{r_\varepsilon})+\int_{B_{r_\epsilon}}|Jv_h|dx,
$$
and notice that, by definition of relaxed functional and \cite[Theorem 3.7]{AD}, 
$$
\liminf_{h\rightarrow+\infty}\mathcal{A}(v_h;B_\radius\setminus B_{r_\varepsilon})\geq\overline{\mathcal{A}}_{L^1}(u;B_\radius\setminus B_{r_\varepsilon})\geq\int_{B_\radius\setminus B_{r_\varepsilon}}\sqrt{1+|\grad w|^2}dx.
$$
Hence
\begin{equation}\label{splitted areaw}
\begin{aligned}
\lim_{h\rightarrow+\infty}{\mathcal{A}(v_h;B_\radius)}&\geq\liminf_{h\rightarrow+\infty}\mathcal{A}(v_h;B_\radius\setminus B_{r_\varepsilon})+\liminf_{h\rightarrow+\infty}\int_{B_{r_\epsilon}}|Jv_h|~dx
\\
&\geq\int_{B_\radius\setminus B_{r_\varepsilon}}\sqrt{1+|\grad w|^2}dx+\liminf_{h\rightarrow+\infty}\int_{B_{r_\epsilon}}|Jv_h|~dx.
\end{aligned}
\end{equation}
To conclude the proof it is then sufficient to show that 
 \begin{equation}\label{singular contribution estimate 2}
\liminf_{h\rightarrow+\infty}\int_{B_{r_\varepsilon}}|Jv_h|dx\geq\pi|
\grado(\varphi)|.
\end{equation}
	
	\noindent
	Define the sequence $w_h:B_\radius\rightarrow\R^2$ as
	\begin{equation}
	w_h(x):=\left\{
	\begin{aligned}
	&v_h(x)  &\quad&\mbox{ if }|x|\leq r_\eps\\
	&\frac{\radius-|x|}{\radius-r_\varepsilon}v_h\left(r_\eps\frac{x}{|x|}\right)+\frac{|x|-r_\eps}{\radius-r_\eps}w\left(r_\eps\frac{x}{|x|}\right)  &&\mbox{ if }r_\eps<|x| < \radius.
	\end{aligned}
	\right.
	\end{equation}
	Then $w_h$ is Lipschitz continuous and interpolates 
$v_h\mres\partial B_{r_\eps}$ and $w\mres\partial B_{r_\eps}$  
in the annulus enclosed by $\partial B_{r_\eps}$ and $\partial B_\radius$.
	Now we show that 
	\begin{equation}\label{contributo regolare di w_k}
	\lim_{h\rightarrow+\infty}\int_{B_\radius\setminus B_{r_\eps}}|Jw_h|dx=0.
	\end{equation}
	Indeed, passing to polar coordinates in ${B_\radius\setminus B_{r_\eps}}$:
\begin{equation}\nonumber
	 w_h(x)=\widetilde{w}_h(\rho,\theta)=
\frac{\radius-\rho}{\radius-r_\varepsilon}\widetilde{v}_h(r_\eps,\theta)+\frac{\rho-r_\eps}{\radius-r_\eps}\widetilde{w}(r_\eps,\theta),
\end{equation}
where
$$
\widetilde{v}_h(r_\eps,\theta):=v_h\left(r_\eps(\cos\theta,\sin\theta)\right))
=((\widetilde{v}_{h})_1(r_\eps,\theta),(\widetilde{v}_{h})_2(r_\eps,\theta)),\quad\widetilde{w}(r_\eps,\theta):=w\left(r_\eps(\cos\theta,\sin\theta)\right)=f(\theta).$$
Making use of \eqref{def_f} and \eqref{lifting system}, we get
	\begin{equation}
	\grad \widetilde{w}_h(\rho,\theta)=
\frac{1}{\radius-r_\eps}
	\begin{pmatrix}
	&-
(\widetilde{v}_{h})_1+f_1   
&(\radius-\rho)\partial_\theta (\widetilde{v}_{h})_1-
(\rho-r_\eps)\Phi' f_2
\\
	&-(\widetilde{v}_{h})_2+f_2
   &(\radius-\rho)\partial_\theta (\widetilde{v}_{h})_2+(\rho-r_\eps)
\Phi'f_1\\
	\end{pmatrix},
	\end{equation}
	where $(\widetilde{v}_{h})_i$, $\partial_\theta (\widetilde{v}_{h})_i$ 
are evaluated at $(r_\eps,\theta)$ for $i=1,2$, and 
$f_1, f_2, \Phi'$ are evaluated at $\theta$.
	Then we can compute the Jacobian determinant of $w_h$ 
in polar coordinates:
	\begin{equation}\nonumber
	\begin{aligned}
	J\widetilde{w}_h(\rho,\theta)
=& 
\frac{1}{(\radius-r_\eps)^2}
\Big[
(\radius-\rho)
\Big\{
(\widetilde{v}_{h})_2\partial_\theta (\widetilde{v}_{h})_1
-\partial_\theta (\widetilde v_{h})_1 f_2
\Big\}
\\
& +(\radius-\rho)
\Big\{\partial_\theta (\widetilde{v}_{h})_2f_1
-(\widetilde{v}_h)_1\partial_\theta (\widetilde{v}_{h})_2
\Big\}
-
(\rho-r_\eps)\Phi'
\Big\{(\widetilde{v}_{h})_1f_1
+(\widetilde{v}_{h})_2f_2-1\Big\}\Big],
\end{aligned}
\end{equation}
where we use also that $f_1^2+f_2^2=1.$
	Thus
	\begin{equation}\label{Jacobian estimate}
	\begin{aligned}
	\int_{B_\radius\setminus B_{r_\eps}}|Jw_h|dx
= &
\int_{r_\eps}^\radius\int_0^{2\pi}|J\widetilde{w}_h|d\rho d\theta
\\
	\leq & C_{\radius,\eps}\int_{r_\eps}^\radius\int_0^{2\pi}|(\widetilde{v}_{h})_2\partial_\theta (\widetilde{v}_{h})_1-\partial_\theta (\widetilde{v}_{h})_1f_2|d\rho d\theta
\\
	&+C_{\radius,\eps}\int_{r_\eps}^\radius\int_0^{2\pi}|(\widetilde{v}_{h})_1\partial_\theta (\widetilde{v}_{h})_2-\partial_\theta (\widetilde{v}_{h})_2f_1|d\rho d\theta
\\
	&+C_{\radius,\eps}\mathrm{lip}(\Phi)\int_{r_\eps}^\radius\int_0^{2\pi}|(\widetilde{v}_{h})_1f_1+(\widetilde{v}_{h})_2f_2
-1|d\rho d\theta,
	\end{aligned}
	\end{equation}
	where $C_{\radius,\eps}$ is a positive constant depending only on $\radius$ and $\eps$.
	Consider the first integral on the right hand side of \eqref{Jacobian estimate}: its integrand is independent of $\rho$, and so 
	\begin{equation}\nonumber
	\begin{aligned}
	 \int_{r_\eps}^\radius\int_0^{2\pi}\left|(\widetilde{v}_{h})_2
\partial_\theta (\widetilde{v}_{h})_1-\partial_\theta (\widetilde{v}_{h})_1f_2(\theta)\right|d\rho d\theta&=(\radius-r_\eps) \int_0^{2\pi}\left|(\widetilde{v}_{h})_2(r_\eps,\theta)-f_2(\theta)\right|\left|\partial_\theta (\widetilde{v}_{h})_1(r_\eps,\theta)\right| d\theta\\
	&\leq C_{\radius,\eps}\|({v}_{h})_2-w_2\|_{L^\infty(\partial B_{r_\eps})}\int_{\partial B_{r_\eps}}\left|\frac{\partial v_h}{\partial s}\right|ds\xrightarrow{k\rightarrow+\infty}0,
	\end{aligned}
	\end{equation}
	where in 
 passing to the limit we used \eqref{strict convergence on circumferencesw}, which implies that the variation of $v_h$ on $\partial B_{r_\eps}$ is 
necessarily equi-bounded and, together with Proposition \ref{strict implies uniform}, that $v_h\rightarrow w$ 
uniformly on $\partial B_{r_\eps}$. For the second integral, the argument is similar. \\
	As for the third one, by the uniform convergence of $(v_h)$ 
to $w$ on $\partial B_{r_\eps}$, we can pass to the limit under the  
integral sign:
	\begin{equation*}
	\int_{r_\eps}^\radius\int_0^{2\pi}|(\widetilde{v}_{h})_1
f_1+(\widetilde{v}_{h})_2f_2-1|d\rho d\theta
	\xrightarrow{h\rightarrow+\infty} 
\int_{r_\eps}^\radius\int_0^{2\pi}|f_1^2+f_2^2-1|d\rho d\theta=0.
	\end{equation*}
	Therefore, \eqref{contributo regolare di w_k} holds. 

Now, we write the Jacobian determinant of $v_h$ on $B_{r_\eps}$ 
in the following way:
	\begin{equation}\label{splitted integral}
	\begin{aligned}
	\int_{B_{r_\varepsilon}}|Jv_h|dx=\int_{B_\radius}|Jw_h|dx-\int_{B_\radius\setminus B_{r_\varepsilon}}|Jw_h|dx.
	\end{aligned}
	\end{equation}
Notice that $w_h=w$ on $\partial B_\radius$, so that (see 
Remarks \ref{rem:too_weak} and \ref{rem:coincidence_of_degrees}) 
\begin{equation}\label{eq:deg_deg_deg}
\mathrm{deg}(w_h,\partial B_\radius)=\mathrm{deg}(w,\partial B_\radius)=\mathrm{deg}(\varphi).
\end{equation}
We may suppose that $v_h$ takes values in $\overline{B}_1$, 
since the limit function $w$ is valued in $\Suno$ (see \cite[Lemma 3.3]{AD}). So $w_h:\overline{B}_\radius\rightarrow \overline{B}_1$ 
is Lipschitz continuous and maps $\partial B_\radius$ into $\partial B_1$.
Then, by \eqref{eq:deg_deg_deg} and \eqref{stima_jacob_grado}, we have
	\begin{equation}\label{degree inequality}
	\begin{aligned}
	\int_{B_\radius}|Jw_h|dx\geq
	\pi|\mathrm{deg}(w,\partial B_\radius)|=\pi|\mathrm{deg}(\varphi)|.
    \end{aligned}
	\end{equation}
	Finally, passing to the lower limit as $h\rightarrow+\infty$ in \eqref{splitted integral}, using \eqref{contributo regolare di w_k} and the previous inequality, 
	we deduce estimate \eqref{singular contribution estimate 2}, 
which concludes the proof.
\end{proof}

\begin{Proposition}[\textbf{Upper bound}]\label{upper bound2}
	Let $w:B_\radius\setminus\{0\}\rightarrow\R^2$ be 
the map defined in  \eqref{def_w}. Then there exists a sequence $(v_k)\subset C^1(B_\radius;\R^2)\cap BV(B_\radius;\R^2)$ such that
	$v_k\rightarrow w$ strictly $BV(B_\radius;\R^2)$ and 
\begin{equation}\label{recover}
\limsup_{k\rightarrow+\infty}\mathcal{A}(v_k;B_\radius)\leq\int_{B_\radius}\sqrt{1+|\grad w|^2}dx + \pi|\grado(\varphi)|.
\end{equation}
\end{Proposition}
\begin{proof}
Although $v_k$ needs to be of class $C^1$, we claim 
that it suffices to build $v_k$ just Lipschitz continuous. 
Indeed, assume that $(v_k)\subset 
W^{1,\infty}(B_\radius;\R^2) \cap C^1(B_\radius; \R^2)$ 
converges to $w$ strictly $BV(B_\radius;\R^2)$ and \eqref{recover} holds. Consider, for all $k \in \mathbb N$, a sequence $(v^k_h)
\subset C^1(B_\radius;\R^2)$ approaching $v_k$ in $W^{1,2}(B_\radius;\R^2)$ 
as $h\rightarrow +\infty$. In particular, we get the $L^1$-convergence of 
all  minors of $\nabla v^k_{h}$ to the corresponding 
ones of $\nabla v_k$. Then, by dominated convergence,
\begin{align}\label{densityenergy}
\lim_{h\rightarrow+\infty}\mathcal{A}(v^k_h;B_\radius)=\mathcal{A}(v_k;B_\radius).
\end{align}
Hence, by a diagonal argument, we find a sequence $(v^k_{h_k})$ 
converging to $w$ strictly $BV(B_\radius;\R^2)$ such that \eqref{recover} holds for $v^k_{h_k}$ in place of $v_k$.

Let us consider the map  $\overline\varphi:\Suno\rightarrow\Suno$ 
given by 
\begin{align}\label{gradostandard}
\overline\varphi(\cos\theta,\sin\theta):=(\cos(d\theta),\sin(d\theta))
\quad {\rm where}~ d:=\grado(\varphi). 
\end{align}
Then
\begin{equation}\label{mappa}
\mathrm{mult}(\overline\varphi)=|\grado(\overline\varphi)|, \qquad
\grado(\overline\varphi)=\grado(\varphi),
\end{equation}
and, in particular, $\mathrm{mult}(\overline\varphi)=|\grado(\varphi)|$.
Moreover, since the maps $\varphi$ and $\overline\varphi$ 
have the same degree, we can construct a Lipschitz homotopy $H:[0,1]\times\Suno\rightarrow\Suno$ between them. Precisely, 
if $\Phi$ and $\overline\Phi$ are Lipschitz
 liftings of $\varphi$ and $\overline\varphi$ respectively, we 
define $\Psi(t,\cdot):=t\Phi(\cdot)+(1-t)\overline\Phi(\cdot)$, which is 
Lipschitz. Hence one defines the map $H(t,\cdot):[0,2\pi)\rightarrow\Suno$ as 
$H(t,\cdot):=(\cos (\Psi(t,\cdot),\sin(\Psi(t,\cdot)))$,
which satisfies 
\begin{equation}\label{omotopia}
H(0,\cdot)=\overline\varphi(\cdot), \qquad
H(1,\cdot)=\varphi(\cdot).
\end{equation}
It remains to show that $H(t,\cdot)$ defines a continuous (and then Lipschitz) map from $\Suno$ to $\Suno$, i.e. that is $2\pi$-periodic: to this aim it is enough to observe that $\Psi(t,2\pi)$ and $\Psi(t,0)$ differ from a constant multiple of $2\pi$ and indeed, recalling \eqref{lifting and degree}, we have
$\Phi(2\pi)-\Phi(0)=2\pi d=\overline\Phi(2\pi)-\overline\Phi(0)$, 
from which easily follows that $\Psi(t,2\pi)-\Psi(t,0)=2\pi d$. \\

We now define the sequence $(v_k)\subset$ Lip$(B_\radius;\R^2)$ as
$v_k(0):=0$, 
\begin{equation}\label{recovery seq}
v_k:=
\begin{cases}
\overline{v}_k & {\rm in}~ B_\frac{\radius}{k} \setminus \{0\},
\\
h_k & {\rm in}~ B_\frac{2\radius}{k}\setminus B_\frac{\radius}{k},
\\
w = \varphi\left(\frac{x}{|x|}\right) & {\rm in}~  B_\radius\setminus B_\frac{2\radius}{k},
\end{cases}
\end{equation}
where 
$$\overline{v}_k(x):=\frac{k}{\radius}|x|\overline\varphi\left(\frac{x}{|x|}\right)\qquad \forall x\in B_\frac{\radius}{k},$$
 and 
$$
h_k(x):=H\left(\frac{k}{\radius}|x|-1,\frac{x}{|x|}\right)\qquad\forall x\in B_\frac{2\radius}{k}\setminus B_\frac{\radius}{k}.
$$
Let us check that 
\begin{align}\label{variation det v_k}
\int_{B_\radius}|J v_k|dx=\pi|d|\qquad \forall k\in\N.
\end{align}
Since $H$ and $w$ take values on $\Suno$,  we have 
$$
\int_{B_\radius\setminus B_\frac{\radius}{k}}|Jv_k|dx=\int_{B_\frac{2\radius}{k}\setminus B_\frac{\radius}{k}}|J h_k|dx+\int_{B_\radius\setminus B_\frac{2\radius}{k}}|J w|dx=0.
$$
Moreover, 
$\mathrm{mult}(\overline{v}_k, B_\frac{\radius}{k},\cdot)$=$\mathrm{mult}(\overline{\varphi})$, and
therefore, by \eqref{eq:det_mult},
$$
\int_{B_\frac{\radius}{k}}|J v_k|dx=\int_{B_\frac{\radius}{k}}|J\overline{v}_k|dx=\int_{B_1}\mathrm{mult}(\overline{v}_k, B_\frac{\radius}{k},y)dy=|B_1|\mathrm{mult}(\overline{\varphi})=\pi|d|.
$$
We now prove that $v_k\rightarrow w$ in $W^{1,p}(B_\radius;\R^2)$ 
for every $p\in[1,2)$. This, in particular, implies 
 the desired strict convergence in $BV$. 
Since $v_k=w$ in $B_\radius\setminus B_\frac{2\radius}{k}$, we have to do the computation on $B_\frac{2\radius}{k}$:
$$
\int_{B_\frac{2\radius}{k}}|v_k-w|^p dx\leq 2^{p-1}
\int_{B_\frac{2\radius}{k}}(|v_k|^p+|w|^p)dx\leq 2^p |B_\frac{2\radius}{k}|\xrightarrow{k\rightarrow +\infty}0.
$$
In addition
$$
|\grad v_k|=|\grad h_k|\leq 2k\,\mathrm{lip}(H)\quad\mbox {a.e. in } B_\frac{2\radius}{k}\setminus B_\frac{\radius}{k},
$$
hence
\begin{equation}
\begin{aligned}
\int_{B_\frac{2\radius}{k}\setminus B_\frac{\radius}{k}}|\grad v_k-\grad w|^pdx&\leq C\left[(2k)^p\mathrm{lip}(H)^p|B_{\frac{2\radius}{k}}|+\int_{B_{\frac{2\radius}{k}}}|\grad w|^pdx\right]\\
&\leq C\left[C\frac{k^p}{k^2}+\int_{B_{\frac{2\radius}{k}}}|\grad w|^pdx\right]\xrightarrow{k\rightarrow+\infty}0,
\end{aligned}
\end{equation}
where $C>0$ is a positive constant independent of $k$.
Finally, setting 
$\overline{w}(x):=\overline\varphi\left(\frac{x}{|x|}\right)$ for $x \in B_\radius \setminus \{0\}$, we have
$$
\grad v_k(x)=
\frac{k}{\ell}|x|\grad \overline{w}(x)+\frac{k}{\ell}\overline{w}(x)\otimes\frac{x}{|x|}
\qquad 
{\rm for~a.e.}~
x \in B_\frac{\radius}{k}.
$$
Whence
\begin{equation}
\begin{aligned}
\int_{ B_\frac{\radius}{k}}|\grad v_k-\grad w|^pdx&\leq C\int_{B_{\frac{\radius}{k}}}\left(k^p|x|^p|\grad \overline{w}|^p+k^p\left|\overline{w}(x)\otimes\frac{x}{|x|}\right|+|\grad w|^p\right)dx\\
&\leq C\left[\int_{B_{\frac{\radius}{k}}}|\grad \overline{w}|^pdx+k^p|B_\frac{\radius}{k}|+\int_{B_{\frac{\radius}{k}}}|\grad w|^pdx\right]\xrightarrow{k\rightarrow +\infty}0.\\
\end{aligned}
\end{equation}
 Now, we easily get \eqref{recover}: upon extracting a (not relabelled)
subsequence such that $(\grad v_k)$ 
converges almost everywhere to $\grad w$, 
by \eqref{variation det v_k} and dominated convergence theorem we have
$$
\limsup_{k\rightarrow +\infty}\mathcal{A}(v_k;B_\radius)
\leq\lim_{k\rightarrow +\infty}\int_{B_\radius}\sqrt{1+|\grad v_k|^2}dx
+\lim_{k\rightarrow +\infty}\int_{B_\radius}|Jv_k|dx=\int_{B_\radius}
\sqrt{1+|\grad w|^2}dx+\pi|d|.
$$
\end{proof}

\begin{Remark}
	In the proof of the upper bound in 
Proposition \ref{upper bound2}
we have shown the $W^{1,p}$ 
convergence of the recovery sequence to the function $w$, for $p\in[1,2)$. 
Hence
	$$
	\overline{\mathcal{A}}_{W^{1,p}}(w;B_\radius)\leq\int_{B_\radius}\sqrt{1+|\grad w|^2}dx + \pi|\grado(\varphi)|.
	$$
	Moreover, since in general 
	$\overline{\mathcal{A}}_{BV}(\cdot\,;B_\radius)\leq\overline{\mathcal{A}}_{W^{1,p}}(\cdot\,;B_\radius)$ for all $p\geq1$,
	we deduce 
	$$
	\overline{\mathcal{A}}_{W^{1,p}}(w;B_\radius)=\int_{B_\radius}\sqrt{1+|\grad w|^2}dx + \pi|\grado(\varphi)|.
	$$
\end{Remark}

\section{Relaxation for maps in $W^{1,1}(\Omega;\Suno)$: Theorem
\ref{relaxed area BV}}\label{sec_general}
In the following lemma we generalize to a generic function in 
$W^{1,1}(B_\radius;\Suno)$ 
the argument used to prove \eqref{strict convergence on circumferencesw}, 
by showing that the strict $BV$ convergence on $B_\radius$ 
is inherited to almost every circumference centered at the origin. 
Unlike 
\eqref{strict convergence on circumferencesw} of 
 Proposition \ref{lower bound2}, 
in this more general context we have to make use of Theorem \ref{Reshetnyak}.

We start to generalize the arguments leading to \eqref{eq:lim_restriction_totvar}. 
\begin{Lemma}[\textbf{Inheritance}]\label{tangential strict convergence}
Let $(v_k)\subset C^1(B_\radius;\R^2)$, $u\in W^{1,1}(B_\radius;\R^2)$, and 
suppose that $v_k\rightarrow u$ strictly $BV(B_\radius;\R^2)$. Then,
 for 
almost every $r\in(0,\radius)$, there exists a subsequence $(v_{k_h})$, depending on $r$, such that 
	$$
	v_{k_h}\mres\partial B_r\rightarrow u\mres\partial B_r \quad\mbox{ strictly } BV(\partial B_r;\R^2).
	$$
\end{Lemma}
\begin{proof}
	The (tangential) variation of the restriction of $u$ on 
$\partial B_r$ is well-defined and finite for almost every $r\in(0,1)$ 
since $u\in W^{1,1}(B_\radius;\R^2)$, and
	$$
	|D(u\mres\partial B_r)|(\partial B_r):=\int_{\partial B_r}\left|\frac{\partial u}{\partial s}\right|ds=\int_{0}^{2\pi}|\partial_\theta\widetilde{u}(r,\theta)|d\theta,
	$$
	where $\widetilde{u}:R := (0,\radius)\times[0,2\pi)\rightarrow\R^2$,
	$\widetilde{u}(\rho,\theta):=u(\rho\cos\theta,\rho\sin\theta)$.
	We compute
	\begin{equation}\label{eq:5.0}
	\int_{R}\left|\partial_\theta\widetilde{u}\right|d\rho d\theta=\int_{B_\radius}
|(\grad u)\tau|dx,
	\end{equation}
	with $\tau(x):=\frac{1}{|x|}(-x_2,x_1), x\neq0$. 
Indeed
\begin{equation}\nonumber
\begin{aligned}
\int_{R}\left|\partial_\theta\widetilde{u}\right|d\rho d\theta&=
\int_0^\radius\!\!\int_0^{2\pi}\left[\sum_{i=1}^2\rho^2\left((\partial_{x_1}u_i)^2
(\sin\theta)^2+(\partial_{x_2}u_i)^2(\cos\theta)^2-2\partial_{x_1}u_i\partial_{x_2}u_i\cos\theta\sin\theta\right)\right]^\frac{1}{2}\!\!\!\! d\rho d\theta
\\
&=\int_{B_\radius}\frac{1}{|x|}\left[\sum_{i=1}^2
\left((\partial_{x_1}u_i)^2x_2^2
+(\partial_{x_2}u_i)^2 
x_1^2-2\partial_{x_1}u_i\partial_{x_2}u_ix_1x_2\right)\right]^\frac{1}{2}\!\!\! dx
\\
&=\int_{B_\radius}\sqrt{|\grad u_1\cdot \tau|^2+|\grad u_2
\cdot\tau|^2}dx
=\int_{B_\radius}|(\grad u)\tau|dx.
\end{aligned}
\end{equation}
In the same way we get
	$$
	\int_{R}\left|\partial_\theta\widetilde{v}_k
\right|d\rho d\theta=\int_{B_\ell}|(\grad v_k)\tau|dx.
	$$
	Thanks to Theorem \ref{Reshetnyak}, with the choices $M=4$,
	$\mathbb S^3\subset\R^4=\R^{2\times2}$, 
$f\in C_b((B_\radius
\setminus\{0\})\times \mathbb S^3)$, 
$$f(x,\sigma):=\sqrt{|\sigma_{{\rm hor}}\cdot \tau(x)|^2+|\sigma_{{\rm vert}}\cdot  \tau(x)|^2},$$ 
where $\sigma\in\mathbb S^3$ and $\sigma_{{\rm hor}}:=
(\sigma_1,\sigma_2), \sigma_{{\rm vert}}:=(\sigma_3,\sigma_4)$,
we obtain
	\begin{equation}\label{Reshetnyak application}
	\lim_{k\rightarrow+\infty}\int_{B_\ell}|(\grad v_k)\tau|dx
=\int_{B_\ell}|(\grad u)\tau|dx.
	\end{equation}
	Now we notice that for almost every $r\in(0,\radius)$ we have
	$$
	v_k\mres\partial B_r\rightarrow u\mres\partial B_r \quad\mbox{in }L^1(\partial B_r;\R^2).
	$$
	Then, since $(v_k\mres\partial B_r)\subset BV(\partial B_r;\R^2)$ for every $r\in(0,\radius)$, by the lower semicontinuity of the variation we get
	\begin{equation}\label{lsc on circumferences}
	\int_{\partial  B_r}\left|\frac{\partial u}{\partial s}\right|ds\leq\liminf_{k\rightarrow+\infty}\int_{\partial B_r}\left|\frac{\partial v_{k}}{\partial s}\right|ds \quad\mbox{for a.e. }r\in(0,\radius).
	\end{equation}
	Integrating with respect to $r$ and by Fatou's lemma, we obtain 
	\begin{equation}\label{integration chain}
	\int_{R}\left|\partial_\theta\tilde{u}
\right|dr d\theta=\int_{0}^\radius \int_{\partial 
 B_r}\left|\frac{\partial u}{\partial s}\right|dsdr\leq\int_0^\ell\liminf_{k\rightarrow+\infty}\int_{\partial B_r}\left|\frac{\partial v_{k}}{\partial s}\right|dsdr\leq\liminf_{k\rightarrow+\infty}\int_{R}\left|\partial_\theta\widetilde{v}_k
\right|drd\theta.
	\end{equation}
	But we notice that, by \eqref{eq:5.0}
and \eqref{Reshetnyak application}, we must have all equalities in \eqref{integration chain}. In particular,
	$$
	\int_{\partial  B_r}\left|\frac{\partial u}{\partial s}\right|ds=\liminf_{k\rightarrow+\infty}\int_{\partial B_r}\left|\frac{\partial v_{k}}{\partial s}\right|ds \quad\mbox{for a.e. }r\in(0,\radius),
	$$
	and we 
conclude extracting a suitable subsequence $(v_{k_h})$ of $(v_k)$ depending on $r$ such that
          $$
           \lim_{h\rightarrow+\infty}\int_{\partial B_r}\left|\frac{\partial v_{k_h}}{\partial s}\right|ds=\liminf_{k\rightarrow+\infty}\int_{\partial B_r}\left|\frac{\partial v_{k}}{\partial s}\right|ds.
          $$
\end{proof}

\begin{Definition}
Let $u\in W^{1,1}(\Omega;\Suno)$ and  $\totvarjac_{W^{1,1}}(u;\Omega) < +\infty$. We set
$$
\totvarjac_{BV}(u;\Omega):=\inf\left\{\liminf_{k\rightarrow+\infty}\totvarjac(v_k;\Omega):
 (v_k)\subset C^1(\Omega,\R^2)\cap BV(\Omega;\R^2), v_k\rightarrow u {\rm ~strictly~} BV\right\}.
$$
\end{Definition}

The proof of Theorem \ref{relaxed area BV} is essentially a 
consequence of the following theorem.

\begin{Theorem}[\textbf{Relaxation of $\totvarjac$ in the
strict convergence}]\label{teo:relaxation_of_TV_in_the_strict_convergence}
Let 
$u\in W^{1,1}(\Omega;\Suno)$ be such that $\totvarjac_{W^{1,1}}(u;\Omega) < +\infty$,
and write ${\rm Det} \grad u$ as in  \eqref{eq:Det_grad_u_sum_deltas}. 
 Then
$$
\totvarjac_{BV}(u;\Omega)
=\pi\sum_{i=1}^\cardsing |d_i|.
$$
In particular, $\totvarjac_{BV}(u;\Omega)
=\totvarjac_{W^{1,1}}(u;\Omega)=|\mathrm{Det}\grad u|(\Omega)$. 
\end{Theorem}

As usual, we divide the proof of Theorem \ref{teo:relaxation_of_TV_in_the_strict_convergence}
into two parts, 
the lower bound (Proposition \ref{lower bound TVBV}) and the upper bound (Proposition \ref{upper bound TVBV}).

\begin{Proposition}[\textbf{Lower bound for $\totvarjac_{BV}$}]
\label{lower bound TVBV}
Let 
$u\in W^{1,1}(\Omega;\Suno)$ be such that $\totvarjac_{W^{1,1}}(u;\Omega) < +\infty$,
and write ${\rm Det} \grad u$ as in  \eqref{eq:Det_grad_u_sum_deltas}. 
Then 
$$
\totvarjac_{BV}(u;\Omega)\geq\pi\sum_{i=1}^\cardsing |d_i|.
$$
\end{Proposition} 
\begin{proof}
According to Theorem \ref{Brezis thm}, 
we choose a radius $\ell>0$ so that the balls $B_\ell({x_i})\subset\Om$, $i=1,\dots,m$, are disjoint. 
Let $(v_k)\subset C^1(\Om;\R^2)$ be such that $v_k\rightarrow u\mbox{ strictly }BV(B_\radius;\R^2)$ and
$$
\lim_{k\rightarrow+\infty}\int_{\Omega}|Jv_k|dx=\totvarjac_{BV}(u;\Omega).
$$
To show the thesis it is sufficient to prove that, for all $i=1,\dots,m$,
$$
\lim_{k\rightarrow+\infty}\int_{B_\radius(x_i)}|Jv_k|dx\geq \pi d_i,
$$
and it suffices to show this inequality for $i=1$.
Let us denote $B_\ell(x_1)$ simply by $B_\ell$. Without loss of generality we may assume $x_1=(0,0)$.
 Since $u\in W^{1,1}(B_\radius;\Suno)$, 
it is $W^{1,1}(\partial B_r;\Suno)$, in particular continuous,
 for almost every $r\in(0,\radius)$.
Thus, we can choose $\overline r>0$ small enough so that $u\mres \partial B_{\overline r}\in W^{1,1}(\partial B_r;\Suno)$. Since the balls $B_\ell(x_i)$, $i=1,\dots,m$, are disjoint, we also have $\grado(u,\partial B_{\overline r},\cdot)=d_1$. 
From Theorem \ref{teo_approx} and Lemma \ref{tangential strict convergence}, 
we get that
\begin{equation}\label{strict convergence on circumferences bis}
\begin{aligned}
&\forall \varepsilon\in(0,\overline r) \quad\exists r_\eps\in(0,\varepsilon)\quad\exists(v_{k_h})\subset(v_k)\quad \exists(u_h)\subset C^\infty(\partial B_{r_\eps};\Suno) \quad\mbox{s.t.}\\
& u\mres \partial B_{r_\varepsilon}\in W^{1,1}(\partial B_{r_\eps};\Suno), \quad u_h\rightarrow u\mres\partial B_{r_\eps} \quad\text{in } W^{1,1}(\partial B_{r_\eps};\Suno),\\
&\text{and } v_{k_h}\mres\partial B_{r_\varepsilon}\rightarrow u\mres\partial B_{r_\varepsilon}\quad\mbox{strictly } BV(\partial B_{r_\varepsilon};\R^2).
\end{aligned}
\end{equation}
In particular, on $\partial B_{r_\eps}$ we have uniform convergence of $(u_h)$ and $(v_{k_h})$  to $u$ by 
Proposition \ref{strict implies uniform}. 
Setting as usual $J v_{k_h} = {\rm det} \grad v_{k_h}$, 
write
$$
\int_{B_{r_\varepsilon}}|Jv_{k_h}|dx=\int_{B_{\overline r}}|Jw_h|dx-\int_{B_{\overline r}\setminus B_{r_\varepsilon}}|Jw_h|dx,
$$
where $w_h\in \text{Lip}(B_{\overline r};\R^2)$ and is given by 
\begin{equation}\label{def of w_k}
w_h(x):=\left\{
\begin{aligned}
&v_{k_h}(x)  &\quad&\mbox{ if }|x|\leq r_\eps\\
&\frac{\overline r-|x|}{\overline r-r_\varepsilon}v_{k_h}\left(r_\eps\frac{x}{|x|}\right)+\frac{|x|-r_\eps}{\overline r-r_\eps}u_h\left(r_\eps\frac{x}{|x|}\right)  &&\mbox{ if }r_\eps<|x|\leq \overline r.
\end{aligned}
\right.
\end{equation}
Now, since $\|v_{k_h}-u_h\|_{L^\infty(\partial B_{r_\eps})}\rightarrow 0$ as $h\rightarrow +\infty$, arguing as in the proof of \eqref{contributo regolare di w_k}
 we have
\begin{equation}\label{Jacobian main computation}
\lim_{h\rightarrow+\infty}\int_{B_{\overline r}\setminus B_{r_\varepsilon}}|Jw_h|dx=0.
\end{equation}
Moreover, from \eqref{def of w_k}  we note that 
\begin{equation}\label{deg of w_k} 
\grado(w_h,\partial B_{\overline r})=\grado(u_h,\partial B_{r_\eps}).
\end{equation}
Thanks to the uniform convergence 
of $(u_h)$ to $u$ on $\partial B_{r_\eps}$, for $h$ large enough, $u_h$ and $u\mres\partial B_{r_\eps}$ must have the same degree
$$\grado(u_h,\partial B_{r_\eps})=\grado(u,\partial B_{r_\eps})=d_1.$$
 Then, arguing as in \eqref{degree inequality}, we obtain that 
$$
\int_{B_{\overline r}}|Jw_h|dx\geq\pi|\grado(w_h,\partial B_{\overline r})|=\pi |d_1|,
$$
for $h\in\N$ sufficiently large.
In conclusion we get
\begin{equation}\label{TV estimate}
\totvarjac_{BV}(u;B_\radius)
=\lim_{h\rightarrow+\infty}\int_{B_\radius}|Jv_{k_h}|dx\geq\liminf_{h\rightarrow+\infty}\int_{B_{r_\eps}}|Jv_{k_h}|dx\geq\liminf_{h\rightarrow+\infty}\int_{B_{\overline r}}|Jw_h|dx\geq \pi|d_1|.
\end{equation}
\end{proof}

\begin{Proposition}[\textbf{Upper bound for $\totvarjac_{BV}$}]\label{upper bound TVBV}
Let 
$u\in W^{1,1}(\Omega;\Suno)$ be such that $\totvarjac_{W^{1,1}}(u;\Omega) < +\infty$,
and write ${\rm Det} \grad u$ as in  \eqref{eq:Det_grad_u_sum_deltas}. 
Then
$$
\totvarjac_{BV}(u;\Omega)\leq\pi\sum_{i=1}^\cardsing |d_i|.
$$
\end{Proposition} 
\begin{proof}
As in the proof of Proposition \ref{lower bound TVBV} 
we choose a radius $\ell>0$ so that the balls $B_\ell({x_i})\subset\Om$, $i=1,\dots,m$, are disjoint. 

We construct a suitable recovery sequence $(v_k)\subset\mathrm{Lip}(\Om;\R^2)$ such that
\begin{equation}\label{recovery job}
\lim_{k \to +\infty}v_k
= u\quad\mbox{in }W^{1,1}(\Om;\R^2)
\end{equation}
and setting $B:=\cup_{i=1}^nB_\ell(x_i)$,
\begin{equation}\label{eq:rec_job}
\lim_{k\rightarrow+\infty}
\int_{B_\radius(x_i)}|Jv_k|dx=\pi | d_i|,\qquad i=1,\dots,m,\qquad \text{ and }\int_{\Om\setminus B}|Jv_k|dx=0.
\end{equation}
As in the proof of Proposition \ref{lower bound TVBV}, we can find $r_1\leq \radius$ so that $u\in W^{1,1}(\partial B_{r_1}(x_i);\R^2)$ and $\grado(u,\partial B_{r_1}(x_i))=d_i$, 
for all $i=1,\dots,m$. For every $k\in\N$, we set $B_k:=\cup_{i=1}^m B_{2^{-k}r_1}(x_i)$. 
By Theorem \ref{approssimanti lisce},  there exists a sequence $\left(u_n^k\right)_{n\in\N}\subset C^{\infty}(\Om\setminus B_k;\Suno)$ such that
\begin{align}\label{4.11}
\lim_{n \to +\infty}
u_n^k = u\quad\quad\mbox{in }W^{1,1}(\Om\setminus B_k;\Suno).
\end{align}
Now, for all $k>1$, we choose 
$r_k\in (2^{-k}r_1,2^{-k+1}r_1)$  such that the following conditions hold: 
for all $i=1,\dots,m$,
\begin{equation}\label{412}
\begin{aligned}
&u\mres\partial B_{r_k}(x_i)\in W^{1,1}(\partial B_{r_{k}}(x_i);\Suno),
\\&
\lim_{n \to +\infty}
\|u^k_n\mres\partial B_{r_{k}}(x_i)- u\mres\partial B_{r_{k}}(x_i)\|_{W^{1,1}(\partial B_{r_{k}}(x_i);\Suno)}=0.
\end{aligned}
\end{equation}
In particular, for all $k>1$ and 
$i=1,\dots,m$, we have 
\begin{equation}\label{413}
\begin{aligned}
\lim_{n \to +\infty}
\|u^k_n\mres\partial B_{r_{k}}(x_i)- u\mres\partial B_{r_{k}}(x_i)\|_{L^\infty(\partial B_{r_{k}}(x_i);\Suno)}= 0,
\end{aligned}
\end{equation}
 thus, using \eqref{eq:deg_restriction}, 
\eqref{412}  and \eqref{eq:grado}, we obtain
\begin{equation}
\begin{aligned}
&|\grado(u^k_n,\partial B_{r_{k}}(x_i))-\grado(u,\partial B_{r_{k}}(x_i))|\\
\leq& \frac{1}{2\pi}\left(\int_{\partial B_{r_{k}}(x_i)}\left|(u_n^k)_1\frac{\partial (u_n^k)_2}{\partial s}-u_1\frac{\partial u_2}{\partial s}\right|ds+\int_{\partial B_{r_{k}}(x_i)}\left|(u_n^k)_2\frac{\partial (u_n^k)_1}{\partial s}-u_2\frac{\partial u_1}{\partial s}\right|ds\right)\longrightarrow0
\end{aligned}
\end{equation}
as $n\rightarrow +\infty$.

 Therefore, there exists $m_k\in\N$ such that, for all $i=1,\dots,m$,
\begin{align}\label{414}
\grado(u^k_n,\partial B_{r_{k}}(x_i))=\grado(u,\partial B_{r_{k}}(x_i))= d_i\qquad\forall n\geq m_k.
\end{align}
Now, using \eqref{4.11} and \eqref{412}, for all $k>1$
there is $\widetilde m_k\in \mathbb N$ such that, for all $i=1,\dots,m$,
\begin{align}
&\|u^k_n-u\|_{W^{1,1}(\Om\setminus (\cup_{i=1}^mB_{r_k}(x_i));\Suno)}\leq\|u^k_n-u\|_{W^{1,1}(\Om\setminus B_k;\Suno)}\leq \frac{1}{k}\qquad\forall n\geq \widetilde m_k,\label{416}	
\\
&\|u^k_n\mres\partial B_{r_{k}}(x_i)- u\mres\partial B_{r_{k}}(x_i)\|_{W^{1,1}(\partial B_{r_{k}}(x_i);\Suno)}\leq\frac1k\qquad\forall n\geq \widetilde m_k.\label{417}
\end{align}
Setting $n_k:=\max\{m_k,\widetilde m_k\}$, we define  $u_k:=u_{n_k}^k$, which satisfies \eqref{414} and \eqref{416}
for all $k>1$. In particular 
\begin{equation}\label{u-u_k convergence}
\lim_{k \to +\infty} 
\|u_k-u\|_{W^{1,1}(\Om\setminus (\cup_{i=1}^mB_{r_k}(x_i));\Suno)}=0.
\end{equation}
 For all $i=1,\dots,m$, let now $\overline\varphi_i:\Suno\rightarrow\Suno$ be the Lipschitz function defined in \eqref{gradostandard} with $d=d_i$, which satisfies
$$
\mathrm{mult}(\overline\varphi_i)=|\grado(\overline\varphi_i)|\quad\mbox{and }\quad\grado(\overline\varphi_i)= d_i;
$$
 Now, for all $i=1,\dots,m$,  $\overline \varphi_i$ 
and $u_k\mres\partial B_{r_{k}}(x_i)$ have the same degree,  
and so there exists a Lipschitz 
homotopy\footnote{To define it 
it suffices to consider two liftings 
of $\overline \varphi_1$ and $u_k (r_k \cdot + x_1)\mres \Suno$,
and linearly interpolate them, as done for $H$ 
in \eqref{omotopia}. Observe that $H_{k,i}$ is Lipschitz since $u_k\mres \partial B_{r_k}(x_i)$ is Lipschitz by the choice of $r_k$.} 
$H_{k,i}:[0,1]\times\Suno\rightarrow\Suno$ such that
$$
H_{k,i}(0,y)=\overline \varphi_i(y),\quad H_{k,i}(1,y)=u_k(r_ky+x_i),\qquad y\in\Suno.
$$
Let us define the sequence $(v_k)\subset\mathrm{Lip}(\Omega;\R^2)$ 
as follows: $v_k:=u_k$ in $\Omega\setminus B$, 
and, for all $i=1,\dots,m$, $v_k(x_i):= 0$ and 
\begin{equation}\label{def of v_k recovery}
v_k(x):=\left\{
\begin{aligned}
&\frac{|x-x_i|}{r_{k+1}}\overline\varphi_i\left(\frac{x-x_i}{|x-x_i|}\right)  &\quad&\mbox{ if }x\in B_{r_{k+1}}(x_i) \setminus 
\{0\},\\
&h_{k,i}(x) &&\mbox{ if }x\in B_{r_{k}}(x_i)\setminus B_{r_{k+1}}(x_i),\\
&u_k(x) && \mbox{ if }x\in B_\radius(x_i)\setminus B_{r_{k}}(x_i),
\end{aligned}
\right.
\end{equation}
where 
$$
h_{k,i}(x):=H_{k,i}\left(\frac{|x-x_i|-r_{k+1}}{r_k-r_{k+1}},\frac{x-x_i}{|x-x_i|}\right) \qquad \forall x\in B_{r_{k}}(x_i)\setminus B_{r_{k+1}}(x_i).
$$
Since $H_{k,i}$ and $u_k$ take values in $\Suno$, we have
$v_k(x) \in \Suno$ 
for $x \in \Om\setminus(\cup_{i=1}^mB_{r_{k+1}}(x_i))$, and so
$$
\int_{\Om\setminus(\cup_{i=1}^mB_{r_{k+1}}(x_i))}|Jv_k|dx=0.
$$
In particular, the second condition in \eqref{eq:rec_job} holds.
Moreover, 
$\mathrm{mult}({v}_k, B_{r_{k+1}}(x_i),\cdot)$=$\mathrm{mult}({\overline \varphi_i})$, and
therefore, by \eqref{eq:det_mult},
$$
\int_{B_{r_{k+1}}(x_i)}|J v_k|dx=\int_{B_1}\mathrm{mult}({v}_k,B_{r_{k+1}}(x_i),y)dy=|B_1|\mathrm{mult}({\overline \varphi_i})=\pi|d_i|,
$$
and also the first condition in \eqref{eq:rec_job} follows.

It remains to show \eqref{recovery job}.
By \eqref{u-u_k convergence} and \eqref{416} we have
\begin{equation}\nonumber
\begin{aligned}
&\int_{\Om}|v_k-u|dx\leq \int_{\Om\setminus (\cup_{i=1}^mB_{r_{k}}(x_i))}|u_k-u|dx+2m|B_{r_{k}}(0)|\rightarrow0\quad\mbox{as }k\rightarrow+\infty,\\
&\int_{\Om\setminus (\cup_{i=1}^mB_{r_{k}}(x_i))}|\grad v_k-\grad u|dx=\int_{\Om\setminus (\cup_{i=1}^mB_{r_{k}}(x_i))}|\grad u_k-\grad u|dx\rightarrow0\quad\mbox{as }
k\rightarrow+\infty.
\end{aligned}
\end{equation}
Now, let us show that, for all $i=1,\dots,m$,
$$
\lim_{k \to +\infty}
\|\grad h_{k,i}\|_{L^1(B_{r_{k}(x_i)}\setminus B_{r_{k+1}}(x_i))}=0.
$$
Let us make the computation for $i=1$, the other cases being identical. 
Set $H_k = H_{k,1}$ and $h_k = h_{k,1}$.
Assume without loss of generality that $x_1=(0,0)$, and 
denote $B_r(x_1)=B_r$.
By definition of  $H_k$  we have
\begin{equation}
\begin{aligned}
\|\partial_t H_k\|_{L^\infty([0,1]\times\Suno)}\leq \|\overline\varphi_1\|_{L^\infty(\Suno)}+\|u_k\|_{L^\infty(\partial B_{r_k})}\leq 2\qquad\forall k\in\N.
\end{aligned}
\end{equation}
Moreover, since $\overline\varphi_1$ is Lipschitz,
\begin{equation}
|\nabla_yH_k(t,y)|\leq |\grad^{\Suno}{\overline\varphi_1}(y)|+r_k|\grad u_k(r_ky)|\leq C+r_k|\grad u_k(r_ky)|.
\end{equation}
We now compute $\grad h_k$ for $x\in B_{r_{k}}\setminus B_{r_{k+1}}$:
$$
\grad h_k(x)=\frac{1}{r_k-r_{k+1}} \partial_tH_k\left(\frac{|x|-r_{k+1}}{r_k-r_{k+1}},\frac{x}{|x|}\right)\otimes\frac{x}{|x|}+\grad_yH_k\left(\frac{|x|-r_{k+1}}{r_k-r_{k+1}},\frac{x}{|x|}\right)\grad\left(\frac{x}{|x|}\right)
$$
and we get
\begin{equation}
\begin{aligned}
&\int_{B_{r_k}\setminus B_{r_{k+1}}}|\grad h_k|dx
\\
\leq&\int_{B_{r_k}\setminus B_{r_{k+1}}}\frac{1}{r_k-r_{k+1}}\left|\partial_tH_k\left(\frac{|x|-r_{k+1}}{r_k-r_{k+1}},\frac{x}{|x|}\right)\right|+\left|\grad_yH_k\left(\frac{|x|-r_{k+1}}{r_k-r_{k+1}},\frac{x}{|x|}\right)\right|\left|\grad\left(\frac{x}{|x|}\right)\right|\!dx\\
\leq& \frac{1}{r_k-r_{k+1}}\|\partial_t H_k\|_{L^\infty}\left|{B_{r_{k}}\setminus B_{r_{k+1}}}\right|+\int_{r_{k+1}}^{r_k}\!\!\!\int_0^{2\pi}\!\!\rho\frac{1}{\rho}\left|\grad_yH_k\left(\frac{\rho-r_{k+1}}{r_k-r_{k+1}},(\cos\theta,\sin\theta)\right)\right|\!d\rho d\theta\\
\leq& C(r_k+r_{k+1})+C(r_k-r_{k+1})+(r_k-r_{k+1})\int_0^{2\pi}r_k|\grad u_k(r_k(\cos\theta,\sin\theta))|d\theta\\
\leq& Cr_k+(r_k-r_{k+1})\int_{\partial B_{r_k}}|\grad u_k|d\mathcal H^1
\leq C\left(r_k+(r_k-r_{k+1})\right)\rightarrow0 \quad\mbox{as }k\rightarrow +\infty,
\end{aligned}
\end{equation}
where we have used \eqref{417} in the last inequality.
Then we conclude
$$
\int_{B_{r_k}\setminus B_{r_{k+1}}}\!\!|\grad v_k-\grad u|dx=\int_{B_{r_k}\setminus B_{r_{k+1}}}|\grad h_k-\grad u|dx\leq\int_{B_{r_k}\setminus B_{r_{k+1}}}|\grad h_k|dx+\int_{B_{r_k}\setminus B_{r_{k+1}}}|\grad u|dx\rightarrow0.
$$
Finally, for  $x\in B_{r_{k+1}}$, we have
$$
\grad v_k(x)=\frac{1}{r_{k+1}}\frac{x}{|x|}\otimes\overline\varphi_1\left(\frac{x}{|x|}\right)+\frac{1}{r_{k+1}}|x|\grad\left(\overline\varphi_1\left(\frac{x}{|x|}\right)\right),
$$
Then, since $\overline\varphi_1$ is Lipschitz,
$$
|\grad v_k(x)|\leq \frac{C}{r_{k+1}},
$$
so we get
$$
\int_{ B_{r_{k+1}}}|\grad v_k-\grad u|dx\leq \frac{C}{r_{k+1}}|B_{r_{k+1}}|+\int_{ B_{r_{k+1}}}|\grad u|dx\rightarrow 0,
$$
and \eqref{recovery job} follows.
\end{proof}

Now, we can prove Theorem \ref{relaxed area BV}.
\begin{proof}
We start with the proof of the lower bound. 
Arguing as in the proof of Proposition \ref{lower bound TVBV}, we may suppose 
 $\cardsing=1$, $\Omega=B_\ell$ and $x_1=(0,0)$.
Let $(v_k)\subset C^1(B_\ell;\R^2)$ be such that $v_k\rightarrow u\mbox{ strictly }BV(B_\ell;\R^2)$ and
$$
\liminf_{k\rightarrow+\infty}\mathcal{A}(v_k;B_\ell)=\lim_{k\rightarrow+\infty}\mathcal{A}(v_k;B_\ell) < +\infty.
$$
Select $r_1>0$ and $ d_1\in\Z$ as in the proof
of Proposition \ref{upper bound TVBV}. Without loss of generality we can suppose that $r_1=\ell$. So we deduce \eqref{strict convergence on circumferences bis} and the 
uniform convergence of $(v_k)$ to $u$ on
almost every circumference in $B_\ell$.
Now write
${\mathcal A}(v_k;B_\ell) = {\mathcal A}(v_k;B_\ell \setminus B_{r_\eps})
+
{\mathcal A}(v_k;B_{r_\eps}) \geq 
{\mathcal A}(v_k;B_\ell \setminus B_{r_\eps})
+ \int_{B_{r_\eps}} \vert J v_k\vert~dx$, 
so that 
\begin{equation}
\begin{aligned}
\lim_{k\rightarrow+\infty}
{\mathcal{A}(v_k;B_\ell)}&\geq\liminf_{k\rightarrow+\infty}\mathcal{A}
(v_{k};B_\ell\setminus B_{r_\varepsilon})+\liminf_{k\rightarrow+\infty}\int_{B_{r_\epsilon}}|Jv_{k}|~dx
\\
&\geq\int_{B_\ell\setminus B_{r_\varepsilon}}\sqrt{1+|\grad u|^2}dx+\liminf_{k\rightarrow+\infty}\int_{B_{r_\epsilon}}|Jv_{k}|~dx.
\end{aligned}
\end{equation}
We now apply \eqref{TV estimate} and next pass to the limit as 
$\eps\rightarrow 0^+$ to get the lower bound in \eqref{BV result}, i.e., 
$$
\liminf_{k \to +\infty}
\mathcal A(v_k; B_\radius) \geq 
 \int_{\Omega}\sqrt{1+|\grad u|^2}dx+
\pi\sum_{i=1}^N|d_i|.
$$

Concerning the proof of the 
upper bound, consider the sequence $(v_k)$ 
defined in \eqref{def of v_k recovery},
which converges to $u$ in $W^{1,1}(\Om; \R^2)$.
 Then, upon extracting a subsequence such that $(\grad v_k)$ converges 
almost everywhere to $\grad u$, by \eqref{eq:rec_job} and dominated convergence we have,
using the inequality $\sqrt{1 + a^2 + b^2  + c^2} \leq \sqrt{1 + a^2 + b^2} + \vert c\vert$ for
$a,b,c \in \R$,
\begin{equation}\nonumber
\begin{aligned}
\limsup_{k\rightarrow + \infty}\mathcal{A}(v_k;B_\ell(x_i))&\leq\lim_{k
\rightarrow +\infty}\int_{B_\ell(x_i)}\sqrt{1+|\grad v_k|^2}dx+\lim_{k\rightarrow
+\infty}\int_{B_\ell(x_i)}|Jv_k|dx\\
&=\int_{B_\ell(x_i)}\sqrt{1+|\grad u|^2}dx+\pi| d_i|,
\end{aligned}
\end{equation}
that leads to
\begin{equation}\nonumber
\begin{aligned}
\limsup_{k\rightarrow + \infty}\mathcal{A}(v_k;\Omega)&\leq\lim_{k
\rightarrow +\infty}\int_{\Omega\setminus\cup_{i=1}^m B_\ell(x_i)}\sqrt{1+|\grad v_k|^2}dx+\limsup_{k\rightarrow + \infty}\mathcal{A}(v_k;\cup_{i=1}^mB_\ell(x_i))\\
&=\int_{\Omega}\sqrt{1+|\grad u|^2}dx+\pi\sum_{i=1}^m| d_i|.
\end{aligned}
\end{equation}
\end{proof}

\begin{Remark}
If $u\in W^{1,p}(\Omega;\Suno)$, $p\in [1,2)$, the recovery sequence defined in \eqref{def of v_k recovery} converges to $u$ in $W^{1,p}(\Omega;\Suno)$ as well. Then, the results of Theorem \ref{teo:relaxation_of_TV_in_the_strict_convergence} and Theorem \ref{relaxed area BV} are still valid if one deals with the relaxation of the area functional with respect to the strong topology of $W^{1,p}(\Omega;\Suno)$.
\end{Remark}

\begin{Remark}
[\textbf{Relaxation in the local uniform convergence 
outside singularities}]
If $u$ is continuous in $\Omega\setminus\{x_1,\ldots,x_m\}$, 
one can relax the area functional with respect to the uniform convergence out of the 
singularities $\{x_i\}$, i.e., we require that  for every compact set $K\subset\Omega\setminus\{x_1,\ldots,x_m\}$  the approximating sequence $(u_k)\subset C^1(\Om;\Suno)$ satisfies
$$
 u_k\rightarrow u\quad \text{in }L^\infty(K),
$$
or, in other words, if $u_k\rightarrow u$ in $L^\infty_{\textrm{loc}}(\Om\setminus \{x_1,\ldots,x_m\};\R^2)$.
Therefore we are led to consider 
\begin{align}
\overline{\mathcal{A}}_{L^\infty}(u;\Omega):=
\inf\Big\{\liminf_{k\rightarrow+\infty}\mathcal{A}(u_k;\Omega):\;
&(u_k)\subset C^1(\Om;\R^2), \;u_k\rightarrow u \mbox{ in }L^1(\Om;\R^2)\nonumber\\
&\mbox{ and } u_k\rightarrow u\text{ in }L^\infty_{\textrm{loc}}(\Om\setminus \{x_1,\ldots,x_m\};\R^2)\Big\}.
\end{align}
It is then possible to show that 
\begin{align}\label{inf_rel}
\overline{\mathcal{A}}_{L^\infty}(u;\Omega)=\int_{\Omega}\sqrt{1+|\grad u|^2}dx+\pi\sum_{i=1}^m|d_i|.
\end{align}
Notice that, if one considers the functional $\totvarjac_{L^\infty}$, obtained by relaxing $\totvarjac$ with this notion of convergence, the counterpart of Theorem \ref{teo:relaxation_of_TV_in_the_strict_convergence} does not hold anymore,
since we cannot guarantee a uniform bound on the $L^1$ norm of  $\nabla v_k$, needed to get
\eqref{Jacobian main computation}; however, we gain such a
 control on $\|\nabla v_k\|_{L^1}$ in the area functional, as soon as the approximating sequence $(v_k)$ has bounded area.

The proof of \eqref{inf_rel} is the same of the one of Theorem \ref{relaxed area BV}, with the difference that  we can
deduce straightforwardly the uniform convergence of $(v_k)$ on almost every circumference in $B_{r_1}$,
without passing through \eqref{strict convergence on circumferences bis}.
\end{Remark}

\section{An extension to symmetric
piecewise constant $BV(\Omega;\Suno)$ maps}\label{sec_triple pt}

In this section we prove Theorem \ref{teo:symmetric_triple-point_map}. 
Let us recall that a symmetric triple point map in $\R^2$  is a map $u=u_T:B_\radius(0)\subset\R^2\rightarrow\Suno$ 
taking three values $\{\alpha,\beta,\gamma\} \subset \Suno$,
vertices of an equilateral triangle, 
 on three non-overlapping $2\pi/3$-angular regions $A,B,C$ 
with common vertex at the origin  and interfaces $a,b,c$ (see Figure \ref{image triple}).
\begin{figure}[h]
    \centering
    \includegraphics[scale=0.5]{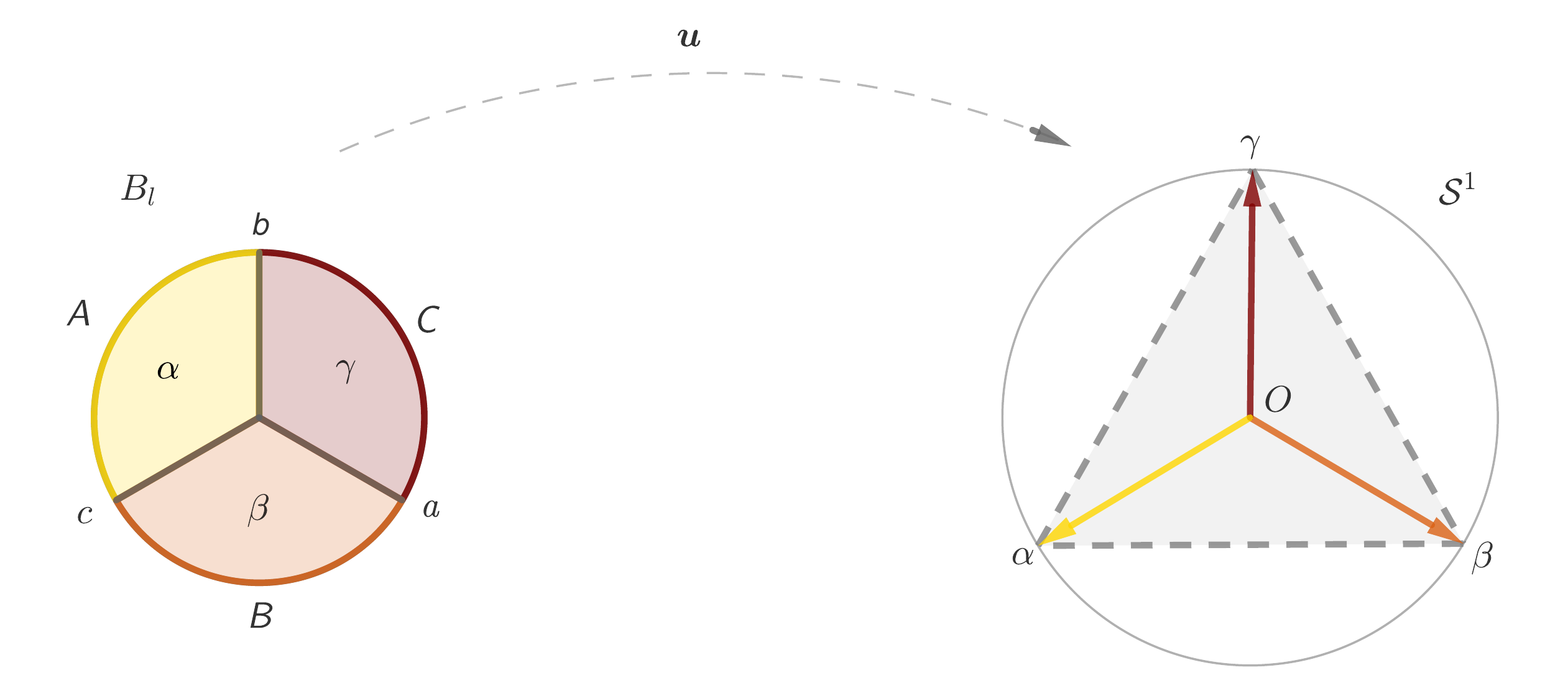}
    \caption{The symmetric triple point map: on the left 
the source disk $B_\radius(0)$, three-sided in the regions $A,B,C$, where $u$ takes the values $\alpha,\beta,\gamma$, depicted in the $\R^2$ target on the right.}
    \label{image triple}
\end{figure}
We denote by $T_{\alpha\beta\gamma}\subset\R^2$ the triangle with vertices $\{\alpha,\beta,\gamma\}$, whose length side is $|\alpha-\beta|=:L=\sqrt{3}$, and by $J_u=a\cup b\cup c$ the jump set of $u$. We have
 $\vert T_{\alpha\beta\gamma}\vert = \frac{\sqrt{3}}{4}L^2=\frac{3\sqrt{3}}{4}$, and  
$|Du|(B_\radius)=L\mathcal H^1(J_u)=3L\radius$.

{\it 
Proof of Theorem \ref{teo:symmetric_triple-point_map}: upper bound}.
For simplicity of notation, in what follows
we write %
$$
\eps {\rm ~in ~place~ of~ } 1/k, 
$$
with
$k \in \mathbb N$.

We construct a recovery sequence $(u^\eps)_\eps\subset\mathrm{Lip}(B_\radius;\R^2)$ as $\eps\rightarrow0^+$.
Let us consider 
the rectangle $$R:=\{(t,s)\in\R^2: t\in(0,\radius), s\in(0,L)\}$$ and,
for $\eps \in (0,\radius)$, the
 functions $m^\eps:R\rightarrow [0,+\infty)$ 
(whose graph is plotted in Figure \ref{m eps}) defined as
\begin{equation}
m^\eps(t,s):=
\begin{cases}
0 & t\in [\eps, \radius]
\\
2\frac{\eps-t}{\eps}\frac{sh}{L} & t \in [0,\eps), \ s \in [0,\frac{L}{2}],
\\
2\frac{\eps-t}{\eps}\frac{(L-s)h}{L} & t \in [0,\eps), \ s\in (\frac{L}{2},L],
\end{cases}
\end{equation}
where $h:=\frac{L}{2\sqrt{3}}=\frac12$. The number $h$ is  
the height of each of the three isosceles 
triangles with common vertex at the origin 
of the target space that decompose $T_{\alpha\beta\gamma}$ 
(see Figure \ref{image triple} right). Let us denote  by $S_\eps^a,S_\eps^b,S_\eps^c$ three tiny stripes 
around $a,b,c$ in $B_\radius$,  of width $\eps$ and length $\radius-\frac{\eps}{2\sqrt{3}}$, 
drawn in Figure \ref{strips}. More explicitely, we have
$$
S_\eps^b:=\left\{(x,y)\in B_\radius:|x|\leq\frac{\eps}{2},y\geq\frac{\eps}{2\sqrt{3}}\right\}
$$
and $S_\eps^a$ ($S_\eps^c$) is obtained by clockwisely rotating $S_\eps^b$ of an angle $\frac{2\pi}{3}$ ($\frac{4\pi}{3}$ respectively) around the origin.

 The idea is to glue $m^\eps$ on each strip in order to build three surfaces embedded in $\R^4$ living in three 
non-collinear copies of $\R^3$, whose total area contribution gives $
\vert T_{\alpha\beta\gamma}\vert$ in the limit $\eps \to 0^+$. 

We introduce  the 
affine diffeomorphism $\psi_\eps:\left[\frac{\eps}{2\sqrt{3}},\radius\right]\rightarrow[0,\radius]$ such that
$$
\psi_\eps'(y)=\frac{\radius}{\radius-\frac{\eps}{2\sqrt{3}}}=:k_\eps\rightarrow1 \quad\mbox{as }\eps\rightarrow 0^+.
$$
Now we can define $u^\eps$ on $S^b_\eps$: we set 
$$
\xi:=\frac{\gamma-\alpha}{L}\in\Suno, \qquad \eta:=-\xi^\perp=\beta,
$$
(where $\xi^\perp$ is the $\frac\pi2$-counterclockwise rotation of $\xi$) 
and
$$
u^\eps(x,y):=
\alpha+ 
\left(\frac{L}{2}+\frac{Lx}{\eps}\right)\xi+m^\eps\left(\psi_\eps(y),\frac{L}{2}+\frac{Lx}{\eps}\right)\eta\qquad\forall (x,y)\in S_\eps^b.
$$
\begin{figure}[t]
    \centering
    \includegraphics[scale=0.5]{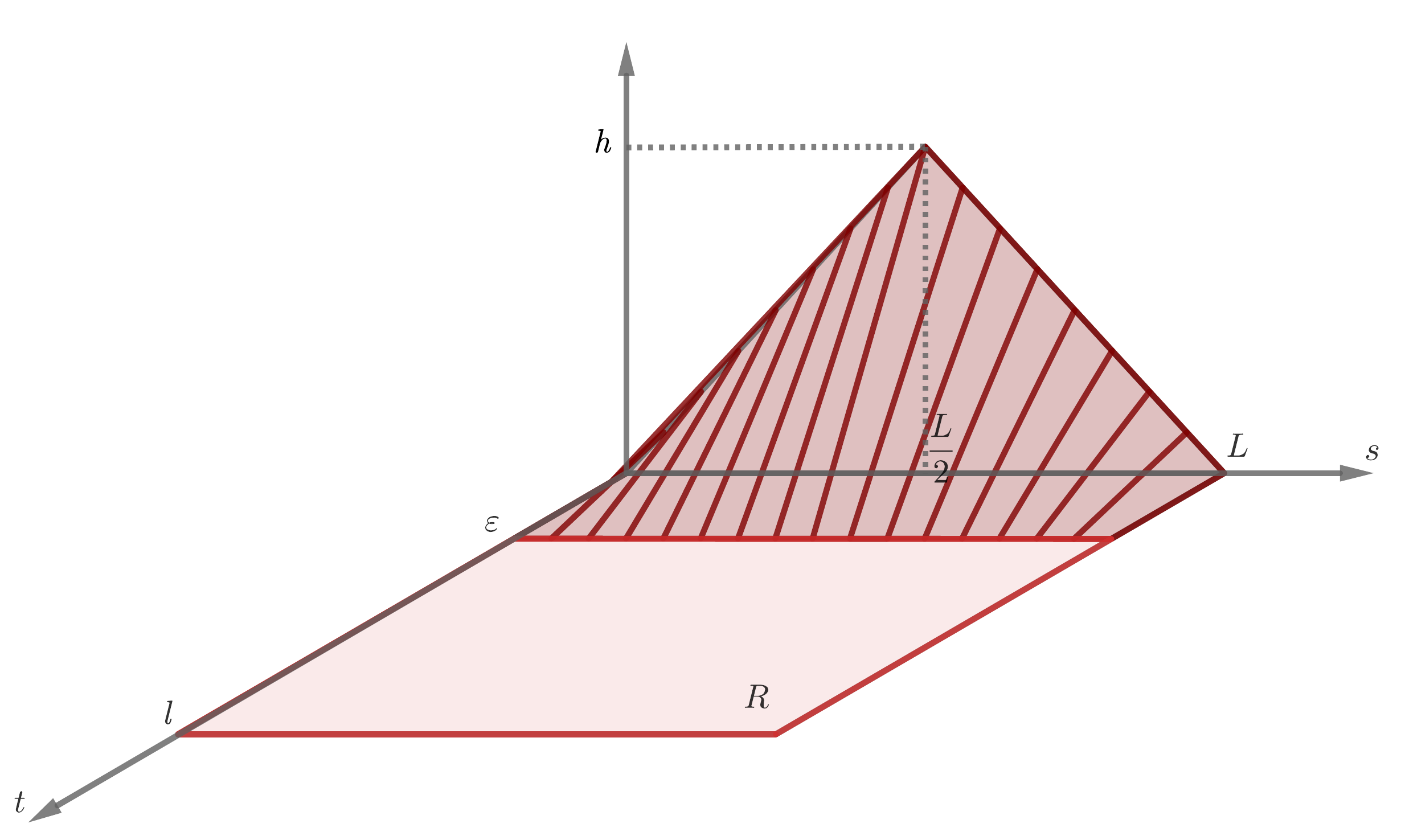}
    \caption{The graph of $m^\eps$ on the rectangle $R$.}
    \label{m eps}
\end{figure}
In a similar way, 
we define $u^\eps$ on $S_\eps^a$ and $S_\eps^c$.
Setting $T^\eps:= \overline{B_{\eps/\sqrt{3}} \setminus (S^a_\eps \cup S^b_\eps \cup S^c_\eps)}$ and $A^\eps:=A\setminus(S_\eps^a\cup S_\eps^b\cup S_\eps^c \cup T^\eps)$, $B^\eps:= B\setminus(S_\eps^a\cup S_\eps^b\cup S_\eps^c \cup T^\eps)$, $C^\eps:=C\setminus(S_\eps^a\cup S_\eps^b\cup S_\eps^c \cup T^\eps)$, we define
\begin{equation}
u^\eps:=
\begin{cases}
\alpha\quad &\mbox{in }A^\eps,
\\
\beta\quad &\mbox{in }B^\eps,
\\
\gamma\quad &\mbox{in }C^\eps.
\end{cases}
\end{equation}
It remains to define $u^\eps$ on the small triangle $T^\eps$. Let us divide it in four triangles $T^a_\eps,T^b_\eps,T^c_\eps,T^0_\eps$ (see Figure \ref{T eps}). So, we set $u^\eps=0$ on $T^0_\eps$ and let $u^\eps$ be the affine function that equals $\alpha$ ($\beta,\gamma$ respectively), in the vertex of $T^\eps$ confining with $A^\eps$ ($B^\eps,C^\eps$ respectively), and equals  $0$ on the edge of $T^0_\eps$.
A direct check shows that the function $u_\eps$ 
is  Lipschitz continuous in $B_\radius$.

Let us compute the area of the graph of $u^\eps$ on $S^b_\eps$: denoting  by $m^\eps_t,m^\eps_s$ the partial derivatives of $m^\eps$, we have
\begin{equation}
\grad u^\eps(x,y)=
\begin{pmatrix}
&\frac{L}{\eps}\xi_1+m^\eps_s(\psi_\eps(y),\frac{L}{2}+\frac{L}{\eps}x)\frac{L}{\eps}\eta_1  &m^\eps_t(\psi_\eps(y),\frac{L}{2}+\frac{L}{\eps}x)k_\eps\eta_1\\[2mm]
&\frac{L}{\eps}\xi_2+m^\eps_s(\psi_\eps(y),\frac{L}{2}+\frac{L}{\eps}x)\frac{L}{\eps}\eta_2  &m^\eps_t(\psi_\eps(y),\frac{L}{2}+\frac{L}{\eps}x)k_\eps\eta_2.
\\
\end{pmatrix}
\end{equation}
Recalling that $\xi\cdot\eta=0$ and
 $|\xi|=|\eta|=1$, we can compute the square of the Frobenius norm of $\grad u^\eps$
\begin{equation}
\begin{aligned}
|\grad u^\eps(x,y)|^2&=\frac{L^2}{\eps^2}\left[\xi_1^2+(m^\eps_s)^2\eta_1^2+2\xi_1\eta_1m^\eps_s+\xi_2^2+(m^\eps_s)^2\eta_2^2+2\xi_2\eta_2m^\eps_s\right]+(m^\eps_t)^2k_\eps^2\eta_1^2+(m^\eps_t)^2k_\eps^2\eta_2^2\\
&=\frac{L^2}{\eps^2}(1+(m^\eps_s)^2)+(m^\eps_t)^2k_\eps^2,
\end{aligned}
\end{equation}
where $m_s^\eps$ and $m_t^\eps$ are evaluated at $\left(\psi_\eps(y),\frac{L}{2}+\frac{L}{\eps}x\right)$.
\begin{figure}[t]
    \centering
    \includegraphics[scale=0.45]{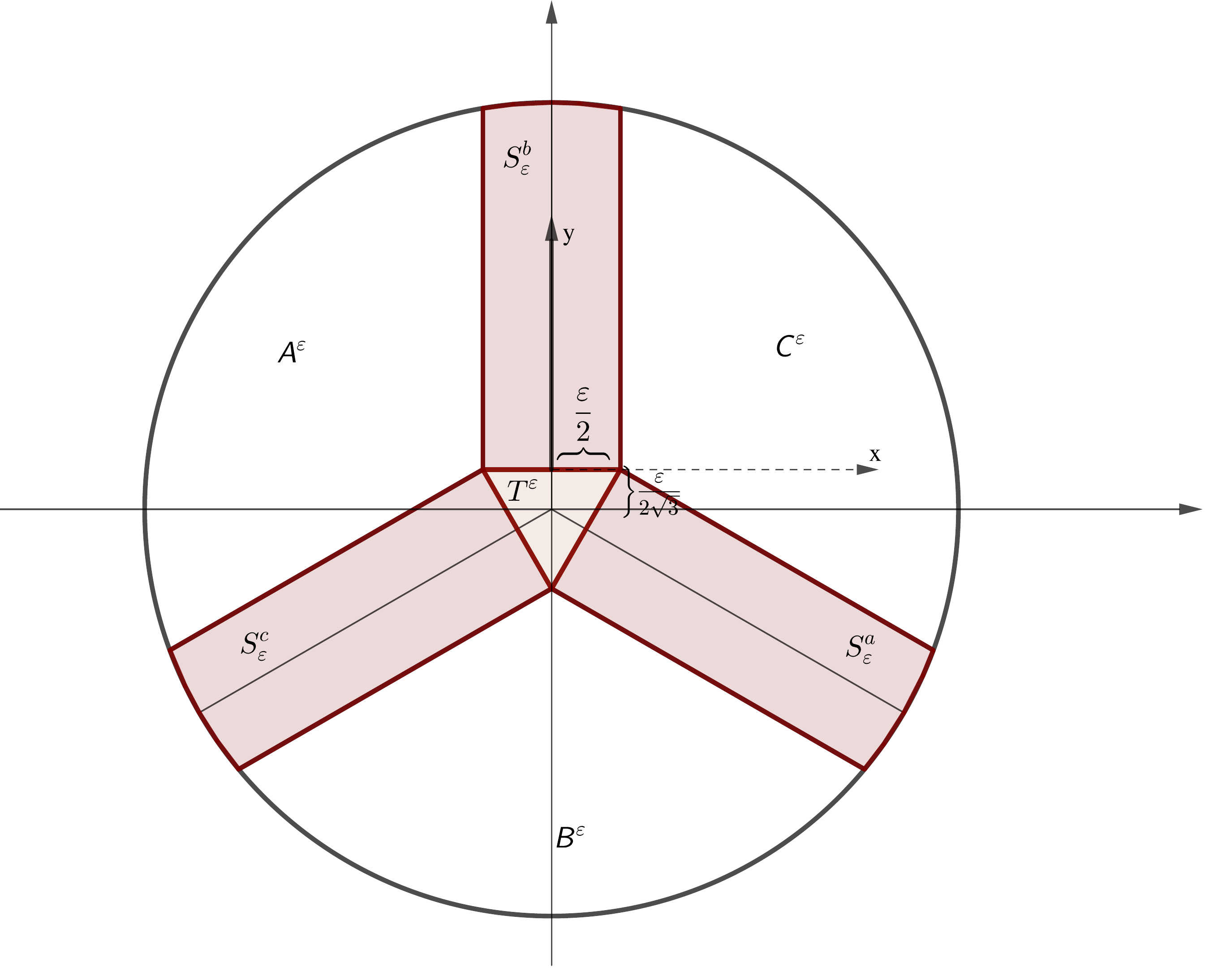}
    \caption{The strips $S_\eps^a,S_\eps^b,S_\eps^c$ and the little triangle $T^\eps$ in the center.}
    \label{strips}
\end{figure}
Moreover, using that   $\xi\cdot\eta^\perp=1$, we have
$$
(\mathrm{det}\grad u^\eps)^2=\frac{k_\eps^2L^2}{\eps^2}\left[(\xi_1\eta_2m^\eps_t+m^\eps_sm^\eps_t\eta_1\eta_2)-(\xi_2\eta_1m^\eps_t+m^\eps_sm^\eps_t\eta_1\eta_2)\right]^2=\frac{k_\eps^2L^2}{\eps^2}(m^\eps_t)^2.
$$
So we have
\begin{equation}
\begin{aligned}
\mathcal{A}(u^\eps;S^b_\eps)&=\int_{S_\eps^b}\sqrt{1+\frac{L^2}{\eps^2}\left(1+(m^\eps_s)^2\right)+(m^\eps_t)^2k_\eps^2+\frac{k_\eps^2L^2}{\eps^2}(m^\eps_t)^2}dxdy\\
&=\frac{L}{\eps}\int_{S^b_\eps}\sqrt{1+m^\eps_s\left(\psi_\eps(y),\frac{L}{2}+\frac{L}{\eps}x\right)^2+m^\eps_t\left(\psi_\eps(y),\frac{L}{2}+\frac{L}{\eps}x\right)^2k_\eps^2\left(1+\frac{\eps^2}{L^2}\right)+O(\eps^2)}dxdy\\
&=\frac{1}{k_\eps}\int_{R\setminus P_\eps}\sqrt{1+m^\eps_s(t,s)^2+m^\eps_t(t,s)^2k_\eps^2\left(1+\frac{\eps^2}{L^2}\right)+O(\eps^2)}dtds,
\end{aligned}
\end{equation}
where in the last equality we have performed the change of variables $$(x,y)=\left(\frac{\eps}{L}\left(s-\frac{L}{2}\right),\psi_\eps^{-1}(t)\right)=:\phi_\eps(t,s)$$ and we have set $P_\eps=R\setminus\phi_\eps^{-1}(S_\eps^b)$.
Notice that $\frac{1}{k_\eps}\rightarrow1$, $k_\eps^2\left(1+\frac{\eps^2}{L^2}\right)\rightarrow1$ as $\eps\rightarrow 0^+$, so that  we get
\begin{equation}\label{inequality for area on strip}
\liminf_{\eps\rightarrow 0^+}\mathcal{A}(u^\eps;S^b_\eps)\leq\int_R1dtds+
\liminf_{\eps\rightarrow 0^+}\int_R|m^\eps_t(t,s)|dtds+
\liminf_{\eps\rightarrow 0^+}\int_R|m^\eps_s(t,s)|dtds.
\end{equation}
Let us compute explicitely the derivatives of $m^\eps$:
\begin{equation}\nonumber
m^\eps_t(t,s)=\left\{
\begin{aligned}
&0 &t>\eps\\
&-2\frac{sh}{\eps L} &t<\eps,s<\frac{L}{2}\\
&-2\frac{(L-s)h}{\eps L} \quad&t<\eps,s>\frac{L}{2}
\end{aligned}
\right.\qquad
m^\eps_s(t,s)=\left\{
\begin{aligned}
&0 &t\geq\eps\\
&2\frac{\eps-t}{\eps}\frac{h}{L} &t<\eps,s<\frac{L}{2}\\
&-2\frac{\eps-t}{\eps}\frac{h}{L} \quad&t<\eps,s>\frac{L}{2}
\end{aligned}
\right.
\end{equation}
Then, we obtain
\begin{equation}\nonumber
\begin{aligned}
&\int_{\{t<\eps,s<\frac{L}{2}\}}|m_t^\eps(t,s)|dtds=\eps\int_{0}^\frac{L}{2}2\frac{sh}{\eps L}ds=\frac{hL}{4}\\
&\int_{\{t<\eps,s>\frac{L}{2}\}}|m_t^\eps(t,s)|dtds=\eps\int_\frac{L}{2}^L2(L-s)\frac{sh}{\eps L}ds=\frac{hL}{4},
\end{aligned}
\end{equation}
so we get
\begin{equation}\label{t derivative}
\int_R|m^\eps_t(t,s)|dtds=\frac{hL}{4}+\frac{hL}{4}=\frac{hL}{2} \quad \forall\eps>0.
\end{equation}
\begin{figure}[t]
    \centering
    \includegraphics[scale=0.5]{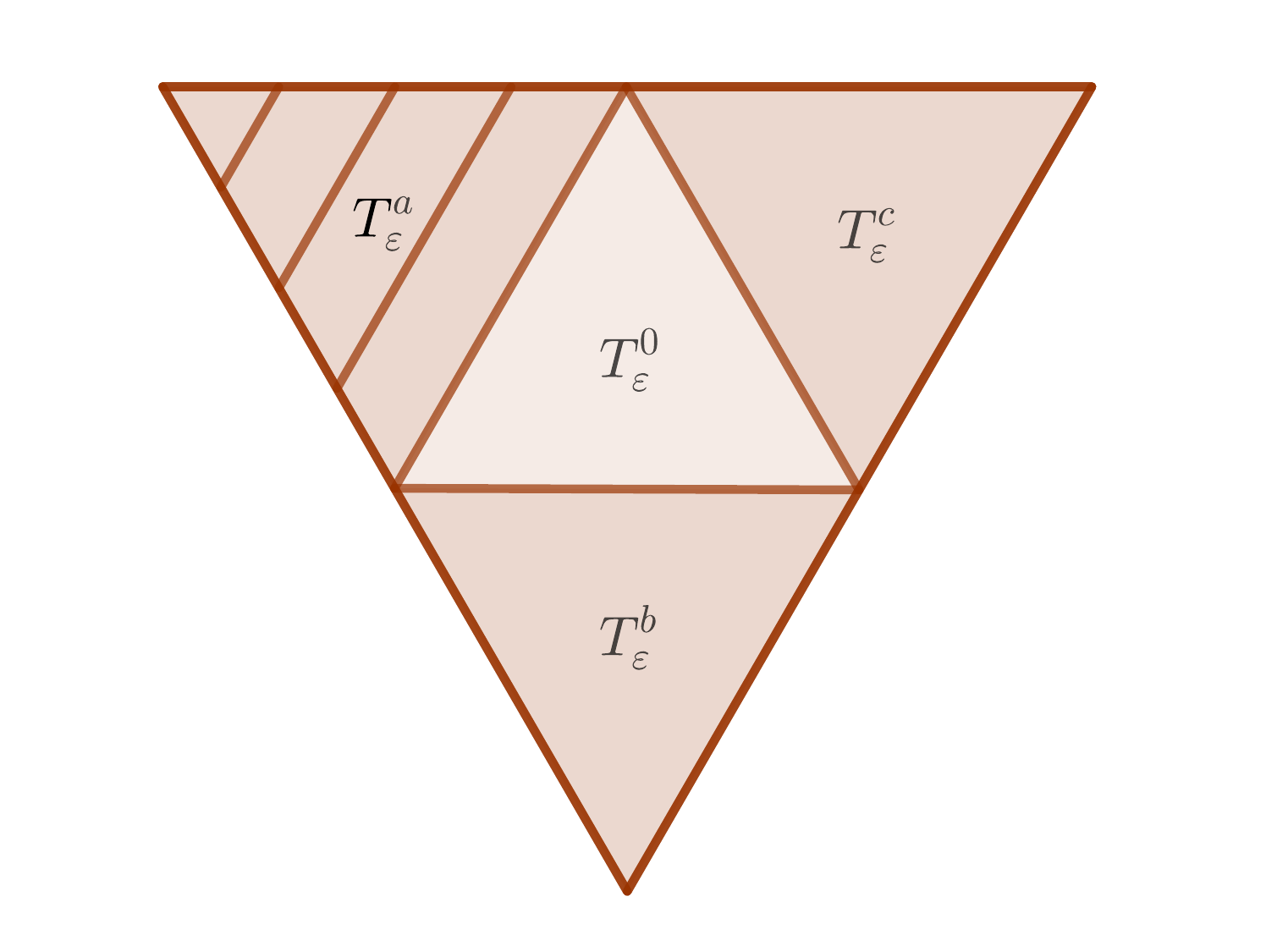}
    \caption{The triangle $T^\eps$ divided further in the four triangles $T^a_\eps,T^b_\eps,T^c_\eps,T^0_\eps$.}
    \label{T eps}
\end{figure}
On the other hand,
$$
\int_{\{t<\eps,s<\frac{L}{2}\}}|m_s^\eps(t,s)|dtds=\int_{\{t<\eps,s>\frac{L}{2}\}}|m_s^\eps(t,s)|dtds=\frac{L}{2}\int_{0}^\eps2\frac{\eps-t}{\eps}\frac{h}{L} ds=O(\eps),
$$
so we get
\begin{equation}\label{s derivative}
\liminf_{\eps\rightarrow 0^+}\int_R|m^\eps_s(t,s)|dtds=0.
\end{equation}
Summarizing, from \eqref{inequality for area on strip} we obtain
$$
\liminf_{\eps\rightarrow 0^+}\mathcal{A}(u^\eps;S^b_\eps)\leq \radius L+\frac{hL}{2}.
$$
In the same way, we can prove that
$$
\liminf_{\eps\rightarrow 0^+}\mathcal{A}(u^\eps;S^a_\eps)=\liminf_{\eps
\rightarrow 0^+}\mathcal{A}(u^\eps;S^c_\eps)\leq \radius L+\frac{hL}{2}.
$$
Clearly, the definition of $u^\eps$ on $A^\eps,B^\eps,C^\eps$ provides that
$$
\lim_{\eps\rightarrow 0^+}\mathcal{A}(u^\eps;A^\eps\cup B^\eps\cup C^\eps)=|B_\radius|=\pi \radius^2.
$$
It remais to show that  the area contribution on $T^\eps$ is infinitesimal: 
first notice that
$$
\mathcal{A}(u^\eps;T_\eps^0)=|T_\eps^0|=O(\eps^2).
$$
Moreover on $T_\eps^a$ (respectively $T_\eps^b, T_\eps^c$) $u^\eps$ is the affine parameterization of the segment $(\alpha,0)$ (respectively $(\beta,0),(\gamma,0)$) of the target space, therefore on $T^\eps\setminus T^0_\eps$ the area integrand has no Jacobian contribution and so is $O(\eps^{-1})$, giving
$$
\mathcal{A}(u^\eps;T_\eps^a)=\mathcal{A}(u^\eps;T_\eps^b)=\mathcal{A}(u^\eps;T_\eps^c)=O(\eps).
$$
Then we have
$$
\mathcal{A}(u^\eps;T^\eps)=\mathcal{A}(u^\eps;T_\eps^0)+\mathcal{A}(u^\eps;T_\eps^a)+\mathcal{A}(u^\eps;T_\eps^b)+\mathcal{A}(u^\eps;T_\eps^c)=O(\eps^2)+O(\eps).
$$
In the end, we conclude
$$
\liminf_{\eps\rightarrow ^+0}\mathcal{A}(u^\eps;B_\radius)\leq\pi \radius^2+3\radius L+3\frac{hL}{2},
$$
where we recognize that 
the last quantity on the right-hand side is exactly $|T_{\alpha\beta\gamma}|$.

As a final step, we have to check that $(u^\eps)$ converges to $u$ strictly $BV(B_\radius;\R^2)$.
Clearly $u^\eps\rightarrow u$ in $L^1(B_\radius;\R^2)$. Let us compute the total variation of $u^\eps$: we have
$$
|Du^\eps|(B_\radius)=|Du^\eps|(S_\eps^a)+|Du^\eps|(S_\eps^b)+|Du^\eps|(S_\eps^c)+|Du^\eps|(T^\eps).
$$
In particular,
$$
|Du^\eps|(T^\eps)\leq\mathcal{A}(u^\eps;T^\eps)\rightarrow0 \quad\mbox{as }
\eps\rightarrow 0^+.
$$
Computing the variation on the strip $S_\eps^b$ (similarly for the other strips) we find
\begin{equation}\nonumber
\begin{aligned}
|Du^\eps|(S_\eps^b)&=\int_{S_\eps^b}\sqrt{\frac{L^2}{\eps^2}\left(1+(m^\eps_s)^2\right)+(m^\eps_t)^2k_\eps^2}dxdy\\
&=\frac{L}{\eps}\int_{S^b_\eps}\sqrt{1+m^\eps_s\left(\psi_\eps(y),\frac{L}{2}+\frac{L}{\eps}x\right)^2+m^\eps_t\left(\psi_\eps(y),\frac{L}{2}+\frac{L}{\eps}x\right)^2k_\eps^2\frac{\eps^2}{L^2}}dxdy\\
&=\frac{1}{k_\eps}\int_{R\setminus P_\eps}\sqrt{1+m^\eps_s(t,s)^2+m^\eps_t(t,s)^2k_\eps^2\frac{\eps^2}{L^2}}dtds.
\end{aligned}
\end{equation}
Then, using \eqref{t derivative} and \eqref{s derivative}, we conclude
\begin{equation}\nonumber
\begin{aligned}
\limsup_{\eps\rightarrow 0^+}|Du^\eps|(S_\eps^b)&\leq\int_R 1dtds+\limsup_{\eps\rightarrow 0^+}\int_R|m^\eps_s(t,s)|dtds+O(\eps)\limsup_{\eps\rightarrow 0^+}\int_R|m^\eps_t(t,s)|dtds= \radius L,
\end{aligned}
\end{equation}
so that
$$
\limsup_{\eps\rightarrow 0^+}|Du^\eps|(B_\radius)\leq 3 \radius L.
$$
By the lower semicontinuity of the variation, we get also 
$$
\liminf_{\eps\rightarrow 0^+}|Du^\eps|(B_\radius)\geq |Du|(B_\radius)=3\radius L,
$$
which shows the desired convergence of $(u^\eps)$ to
$u$ strictly $BV(B_\radius;\R^2)$.
\qed

\medskip

Before proving the lower bound, similarly to 
Lemma 
\ref{tangential strict convergence}, 
we show 
that the strict $BV$
convergence is inherited to almost
every circumference centered at the origin.

\begin{Lemma}[\textbf{Inheritance}]\label{strict convergence on circumference for triple point}
Lemma 
\ref{tangential strict convergence} holds with $\triplemap$ in place
of $u$.
\end{Lemma}

\begin{proof}
Let $\rho<\radius$ and $u$ be the triple point map; clearly
\begin{equation}\label{Var on circumference}
|D(u\mres\partial B_\rho)|(\partial B_\rho)=3L.
\end{equation}
On the other hand, since $(v_k)$ converges to $u$ in $L^1$, 
for almost every $\rho<\radius$ we have
$v_k\mres\partial B_\rho\rightarrow u\mres\partial B_\rho \quad\mbox{in }L^1(\partial B_\rho;\R^2)$,
and by lower semicontinuity we infer that
\begin{equation}
\label{eq:semicontinuity_on_partial_B_rho}
|D(u\mres\partial B_\rho)|(\partial B_\rho)\leq\liminf_{k\rightarrow +\infty}\int_{\partial B_\rho}\left|\frac{\partial v_k}{\partial s}\right|ds \qquad\mbox{for a.e. } \rho<\radius.
\end{equation}
Integrating with respect to $\rho\in(0,\radius)$, by \eqref{Var on circumference} 
and Fatou's lemma, we have
\begin{equation}\label{chain of inequalities}
|Du|(B_\radius)=3\radius L=\int_0^\radius|D(u\mres\partial B_\rho)|(\partial B_\rho)d\rho\leq\int_0^\radius \liminf_{k\rightarrow +\infty}\int_{\partial B_\rho}\left|\frac{\partial v_k}{\partial s}\right|dsd\rho\leq\liminf_{k\rightarrow +\infty}\int_{B_\radius}|\nabla v_k|dx.
\end{equation}
By assumption, $(v_k)$ converges to $u$ strictly $BV(B_\radius;\R^2)$, so we have all equalities in \eqref{chain of inequalities}, in particular,
using \eqref{eq:semicontinuity_on_partial_B_rho},
$$
|D(u\mres\partial B_\rho)|(\partial B_\rho)=\liminf_{k\rightarrow +\infty}
\int_{\partial B_\rho}\left|\frac{\partial v_k}{\partial s}\right|ds \qquad\mbox{for a.e. } \rho<\radius.
$$
Upon extracting a suitable subsequence $(v_{k_h})$ depending on $\rho$ we get the conclusion.
\end{proof}


{\it Proof of Theorem \ref{teo:symmetric_triple-point_map} (lower bound)}. 
Let $(v_k)\subset C^1(B_\radius;\R^2)$ be a recovery sequence, i.e., 
$$
v_k\rightarrow u\quad\mbox{strictly }BV(B_\radius;\R^2)\qquad\mbox{and}\qquad\lim_{k\rightarrow +\infty}\mathcal{A}(v_k;B_\radius)=
\overline{\mathcal{A}}_{BV}(u;B_\radius).
$$
Fix $\rho\in (0,\radius)$ and a subsequence $(v_{k_h})$ of $(v_k)$ whose restriction to $\partial B_\rho$ converges to $u\mres\partial B_\rho$ strictly $BV(\partial B_\rho;\R^2)$, as in Lemma \ref{strict convergence on circumference for triple point}. For simplicity, let us still denote $v_{k_h}$ by $v_k$. 

Let us  split the area functional as
$$
\mathcal{A}(v_k;B_\radius)=\mathcal{A}(v_k;B_\radius\setminus B_\rho)+\mathcal{A}(v_k; B_\rho).
$$
On $B_\radius\setminus B_\rho$ we still have $L^1$-convergence 
of $(v_k)$ to $u$, but $u\mres( B_\radius\setminus B_\rho)$ has no triple points, so by Theorem 3.14 of \cite{AD},
\begin{equation*}
\begin{aligned}
\liminf_{k\rightarrow +\infty}\mathcal{A}(v_k;B_\radius\setminus B_\rho)&\geq\overline{\mathcal{A}}_{L^1}(u;B_\radius\setminus B_\rho)=\int_{B_r\setminus B_\rho}|\sqrt{1+|\nabla u|^2}dx+|D^ju|(B_\radius\setminus B_\rho)\\
&=|B_\radius\setminus B_\rho|+3L(\radius-\rho)=\pi(\radius^2-\rho^2)+3L(\radius-\rho).
\end{aligned}
\end{equation*}
Therefore
\begin{equation}\label{lower bound for the limit}
\begin{aligned}
\lim_{k\rightarrow +\infty}\mathcal{A}(v_k;B_\radius)\geq & 
\liminf_{k\rightarrow +\infty}\mathcal{A}(v_k;B_\radius\setminus B_\rho)
+\liminf_{k\rightarrow +\infty}\mathcal{A}(v_k; B_\rho)
\\
\geq& \pi(\radius^2-\rho^2)+3L(\radius-\rho)+\liminf_{k\rightarrow +\infty}\int_{B_\rho}|Jv_k|dx,
\end{aligned}
\end{equation}
where as usual $J v_k := {\rm det} \grad v_k$.

Let us prove that
\begin{equation}\label{lower bound for the Jacobian liminf}
\liminf_{k\rightarrow +\infty}\int_{B_\rho}|Jv_k|dx\geq 
\vert 
T_{\alpha\beta\gamma}\vert,
\end{equation}
from which 
the lower bound in \eqref{eq:relaxed_triple-point} is obtained
 by passing to the limit as $\rho\rightarrow 0^+$ in \eqref{lower bound for the limit}.
Now we observe that, since $v_k$ is Lipschitz on $B_\rho$, it satisfies the following identity (see \eqref{int_det_bordo})
$$
\int_{B_\rho}Jv_kdx= \frac{1}{2}\int_{\partial B_\rho}\Big((v_k)_1\frac{\partial (v_k)_2}{\partial s}-(v_k)_2\frac{\partial (v_k)_1}{\partial s}\Big)ds\quad\forall k\in\N.
$$

\noindent
Let us parametrize $\partial B_\rho$ from $[0,2\pi)$ and set $\widetilde{v}_k(t):=v_k({s}(t))$ for $t \in [0,2\pi)$; then
\begin{equation}\nonumber
\begin{aligned}
(\dot{\widetilde{v}}_{k})_i(t)=\frac{d}{dt}(v_{k})_i
({s}(t))=\rho\frac{\partial (v_k)_i}{\partial s}(s(t)), \qquad i=1,2.
\end{aligned}
\end{equation}
Thus we get
$$
 \int_{\partial B_\rho}\Big((v_k)_1\frac{\partial (v_k)_2}{\partial s}-(v_k)_2\frac{\partial (v_k)_1}{\partial s}\Big)ds=\int_0^{2\pi}\left(({\widetilde{v}}_{k})_1(t)(\dot{\widetilde{v}}_{k})_2(t)-({\widetilde{v}}_{k})_2(t)(\dot{\widetilde{v}}_{k})_1(t)\right)dt.
$$
Denoting  ${\widetilde{v}}_{k}(t)$ 
simply by ${{v}}_{k}(t)$, we can write
$$
\int_{B_\rho}Jv_kdx=\frac{1}{2}\int_0^{2\pi}\left({({v}}_{k})_1(t)(\dot{{v}}_{k})_2(t)-({{v}}_{k})_2(t)(\dot{{v}}_{k})_1(t)\right)dt.
$$
To show \eqref{lower bound for the Jacobian liminf} it is sufficient to prove that 
\begin{equation}\label{eq:claim for the Jacobian limit}
\liminf_{k\rightarrow +\infty}\frac{1}{2}\int_0^{2\pi}
\left({({v}}_{k})_1(t)(\dot{{v}}_{k})_2(t)-({{v}}_{k})_2(t)(\dot{{v}}_{k})_1(t)\right)dt
\geq \vert T_{\alpha\beta\gamma} \vert,
\end{equation} 
since obviously
$$
\int_{B_\rho}|Jv_k|dx\geq\left|\int_{B_\rho}Jv_kdx\right|.
$$
In order to show \eqref{eq:claim for the Jacobian limit}, denote by $\theta_1\in[0,2\pi)$ (respectively $\theta_2,\theta_3$) the angle of the 
middle point of the arc $C \cap \partial B_\rho$ (respectively $A \cap 
\partial B_\rho$, $B \cap \partial B_\rho$) and write
\begin{equation}
\begin{aligned}
& \frac{1}{2}\int_0^{2\pi}\left(({v}_{k})_1(t)(\dot{{v}}_{k})_2(t)-
({{v}}_{k})_2(t)(\dot{{v}}_{k})_1(t)\right)dt 
\\
&= \frac{1}{2}\int_{\theta_1}^{\theta_2}\left({({v}}_{k})_1(t)(\dot{{v}}_{k})_2(t)-
({{v}}_{k})_2(t)(\dot{{v}}_{k})_1(t)\right)dt
\\
&\;\;+\frac{1}{2}\int_{\theta_2}^{\theta_3}\left({({v}}_{k})_1(t)(\dot{{v}}_{k})_2(t)-({{v}}_{k})_2(t)(\dot{{v}}_{k})_1(t)\right)dt\\
&\;\;+\frac{1}{2}\int_{\theta_3}^{\theta_1}\left({({v}}_{k})_1(t)(\dot{{v}}_{k})_2(t)-({{v}}_{k})_2(t)(\dot{{v}}_{k})_1(t)\right)dt.
\end{aligned}
\end{equation}
Notice that, as a consequence of  Lemma \ref{strict convergence on circumference for triple point},  $v_k$ converges to $u$ strictly $BV([\theta_1,\theta_2];\R^2)$. Furthermore, by restricting $v_k$ to $[\theta_1,\theta_1+\delta]$, for a small $\delta>0$, as a consequence of Proposition \ref{strict implies uniform} we see that $v_k$ converges uniformly to $v\equiv \gamma$ on $[\theta_1,\theta_1+\delta]$. In particular we have
$$\lim_{k\rightarrow \infty}v_k(\theta_1)=\gamma.$$
Similarly $v_k$ will tend to $\alpha$ and $\beta$ in $\theta_2$ and $\theta_3$, respectively.
 We set
$$
L_k:=\int_{\theta_1}^{\theta_2}|\dot{v}_k(t)|dt,\qquad z(t)=z_k(t):=\int_{\theta_1}^t|\dot{v}_k(\tau)|d\tau, \quad t \in [\theta_1,\theta_2].
$$
Denoting by $t(z)$ the inverse of $z(t)$, we define $w_k:[0,L_k]\rightarrow\R^2$ as
$$
w_k(z)=v_k(t(z)).
$$
Then we have
$$
w_k'(z)=\dot{v}_k(t(z))\frac{dt}{dz}=\frac{\dot{v}_k(t(z))}{|\dot{v}_k(t(z))|},\quad dt=\frac{1}{|\dot{v}_k(t(z))|}dz.
$$
Thus
\begin{equation}\label{change of variable by reparameterization}
\begin{aligned}
\frac{1}{2}\int_{\theta_1}^{\theta_2}\left(({v}_{k})_1(t)(\dot{{v}}_{k})_2(t)-({{v}}_{k})_2(t)(\dot{{v}}_{k})_1(t)\right)dt=\frac{1}{2}\int_0^{L_k}\left((w_{k})_1(z)(w_k')_{2}(z)-(w_k)_{2}(z)(w_k')_{1}(z)\right)dz.
\end{aligned}
\end{equation}
We also have 
$$
\lim_{k\rightarrow +\infty} L_k=
\lim_{k\rightarrow +\infty} \int_{\theta_1}^{\theta_2}|\dot{v}_k(t)|dt=|Du|\mres\{y\in\partial B_\rho: \mathrm{arg}(y)\in[\theta_1,\theta_2]\}=|\gamma-\alpha|=L.
$$ 
Then, $(w_k)_k$ is uniformly Lipschitz continuous on $[0,L_k]$. We 
further reparametrize it on $[0,L]$ by a multiple of the arc length parameter. Still denoting the obtained function by $(w_k)_k$, we see that 
 $w_k$ is uniformly bounded in $W^{1,\infty}([0,L];\R^2)$ so, upon 
passing to a (not relabelled) subsequence, we have
$$
w_k\stackrel{*}{\rightharpoonup} w\quad \mbox{w}^*\mbox{-}W^{1,\infty}([0,L];\R^2),
$$
for some $w\in W^{1,\infty}([0,L];\R^2)$.
Hence, we can pass to the limit in \eqref{change of variable by reparameterization}, which now reads
\begin{equation}\label{passage to limit}
\frac{1}{2}\int_0^{L}\left((w_{k})_1(z)(w_k')_{2}(z)-(w_k)_{2}(z)(w_k')_{1}(z)\right)dz
\xrightarrow{k\rightarrow +\infty} \frac{1}{2}\int_0^{L}\left({{w}}_{1}(z)w'_{2}(z)-{{w}}_{2}(z)w'_{1}(z)\right)dz.
\end{equation}
Recalling that 
\begin{equation*}
\begin{aligned}
&w(0)=\lim_{k\rightarrow +\infty}w_k(0)=\lim_{k\rightarrow +\infty}v_k(\theta_1)=\gamma,
\\
&w(L)=\lim_{k\rightarrow +\infty}w_k(L)=\lim_{k\rightarrow +\infty}w_k(L_k)=\lim_{k\rightarrow +\infty}v_k(\theta_2)=\alpha,
\end{aligned}
\end{equation*}
we see that  $w$ is a $1$-Lipschitz curve on $[0,L]$
 starting from $\gamma$ and ending at $\alpha$;
 therefore it must coincide with the unit
 speed parameterization of the 
segment connecting $\gamma$ to $\alpha$, i.e.,
$$
w(z)=\gamma+\frac{\alpha-\gamma}{L}z.
$$
So, we can easily compute the limit integral in \eqref{passage to limit}:
\begin{equation}\nonumber
\begin{aligned}
\frac{1}{2}\int_0^{L}\left({{w}}_{1}(z)w'_{2}(z)-{{w}}_{2}(z)w'_{1}(z)\right)dz&=-\frac{1}{2}\int_0^{L}\left(\gamma+\frac{\alpha-\gamma}{L}z\right)\cdot\frac{(\alpha-\gamma)^\perp}{L}dz=-\frac{1}{2}\gamma\cdot(\alpha-\gamma)^\perp\\
&=\frac{1}{2}(\gamma_1\alpha_2-\gamma_2\alpha_1)=\vert T_{\alpha0\gamma}\vert,
\end{aligned}
\end{equation}
where $T_{\alpha0\gamma}$ is the triangle with vertices $\alpha$, $\gamma$ and the origin $0$. We conclude that 
$$
\lim_{k\rightarrow+\infty}\frac{1}{2}\int_{\theta_1}^{\theta_2}\left(({v}_{k})_1(t)(\dot{v}_k)_{2}(t)-({v}_{k})_2(t)(\dot{v}_k)_{2}(t)\right)dt= \vert T_{\alpha0\gamma}\vert.
$$
In a similar way, one can prove that
\begin{equation}\nonumber
\begin{aligned}
&\lim_{k\rightarrow +\infty}\frac{1}{2}\int_{\theta_2}^{\theta_3}\left(({v}_{k})_1(t)(\dot{v}_k)_{2}(t)-({v}_{k})_2(t)(\dot{v}_k)_{2}(t)\right)dt= \vert T_{\alpha0\beta}\vert 
\\
&\lim_{k\rightarrow +\infty}\frac{1}{2}\int_{\theta_3}^{\theta_1}\left(({v}_{k})_1(t)(\dot{v}_k)_{2}(t)-({v}_{k})_2(t)(\dot{v}_k)_{2}(t)\right)dt= \vert T_{\beta0\gamma}\vert,
\end{aligned}
\end{equation}
and \eqref{eq:claim for the Jacobian limit} follows.
\qed

\begin{Remark}
A result similar to Theorem \ref{teo:symmetric_triple-point_map} holds, up to trivial modifications, when $u:B_\radius(0) \to \Suno$ is a symmetric $n$-junction map,
taking (in the order) the values $\alpha_1,\ldots,\alpha_n$ vertices of the regular $n$-gon $P_{\alpha_1\cdots\alpha_n}$ inscribed in the 
unit circle, on $n$ non-overlapping $2\pi/n$-angular regions
with common vertex at the origin. In formulas, let $L$ be the side of $P_{\alpha_1\cdots\alpha_n}$ and $h$ be the height of each isosceles triangle 
that decomposes $P_{\alpha_1\cdots\alpha_n}$, then there holds the following 
\end{Remark}
\begin{Corollary}\label{n point area}
Let $u:B_\radius(0) \to \Suno$ be a symmetric $n$-junction map. Then
$$
\overline{\mathcal{A}}_{BV}(u,B_\radius )=|B_\radius |+|Du|(B_\radius)+|P_{\alpha_1\cdots\alpha_n}|=\pi \radius^2+nL\radius+\frac{n}{2}hL.
$$
\end{Corollary}

\subsection*{Acknowledgements}
We acknowledge the financial support of the GNAMPA of INdAM (Italian institute of high mathematics).



\begin{thebibliography}{AA} 
\addcontentsline{toc}{chapter}{Bibliografia} 


\bibitem{AD}
E. Acerbi and G. Dal Maso, 
{\it New lower semicontinuity results for
polyconvex integrals}, Calc. Var. Partial Differential Equations
{\bf 2} (1994), 329-371.

\bibitem{AFP}
L. Ambrosio, N. Fusco and D. Pallara, ``Functions of Bounded Variation and 
Free Discontinuity Problems'', Oxford Mathematical
Monographs, Oxford University Press, New York, 2000.

\bibitem{BEPS}
G. Bellettini, A. Elshorbagy, M. Paolini and  R. Scala,
{\it 
On the relaxed area of the graph of discontinuous maps from the plane to the 
plane taking three values with no symmetry assumptions}, 
Ann. Mat. Pura Appl. {\bf 199} 445–477 (2020).

\bibitem{BES}{ G.~Bellettini, A.~Elshorbagy and R.~Scala}
{\it The $L^1$-relaxed area of the graph of the vortex map}, 
submitted. Preprint arXiv 2107.07236, https://arxiv.org/abs/2107.07236 (2021).

\bibitem{BMS}{ G.~Bellettini, R. Marziani and R.~Scala},
{\it 
A non-parametric Plateau problem with partial free boundary},  submitted. Preprint arXiv 2201.06145, https://arxiv.org/abs/2201.06145 (2022).


\bibitem{BP}
G. Bellettini and M. Paolini,
{\it  On the area of the graph of a singular
map from the plane to the plane taking three values},
Adv. Calc.
Var. {\bf 3} (2010), 371-386.

\bibitem{BePaTe:15}
G. Bellettini, M. Paolini and L. Tealdi,
{\it On the area of the graph of a piecewise smooth map from the
	plane to the plane with a curve discontinuity},
 ESAIM: Control Optim. Calc. Var.
{\bf 22} (2015), 29--63.

\bibitem{BePaTe:16}
G. Bellettini, M. Paolini and L. Tealdi,
{\it Semicartesian surfaces and  the relaxed area of
	maps from the plane to the plane with a line discontinuity},
Ann. Mat. Pura Appl.
{\bf 195} (2016), 2131--2170.



\bibitem{BMP}
H. Brezis, P. Mironescu
and A. Ponce. $W^{1,1}$-maps with values into $\Suno$. \textit{Geometric analysis of PDE and several complex variables}, 69–100, Contemp. Math., 368, Amer. Math. Soc., Providence, RI (2005).


\bibitem{DalMaso:80} 
G. Dal Maso,
\emph{Integral representation on $BV(\Omega)$ of $\Gamma$-limits
	of variational integrals},
Manuscripta Math. {\bf 30} (1980), 387--416.

\bibitem{DeGiorgi:92} E. De Giorgi, 
\emph{On the relaxation of functionals defined on cartesian manifolds},
In ``Developments in Partial Differential Equations and Applications
in Mathematical Physics'' (Ferrara 1992),
Plenum Press, New York (1992).



\bibitem{DP}
G. De Philippis,
{\it Weak notions of Jacobian determinant and relaxation},
ESAIM: Control Optim. Calc. Var.
{\bf 18} (2012), 181--207. 

\bibitem{GMS}
M. Giaquinta, G. Modica and J. Sou\v{c}ek, ``Cartesian Currents in the
Calculus of Variations I'', vol. 37, Springer-Verlag, Berlin, 1998.

\bibitem{GiMoSu:98_2} { M. Giaquinta, G. Modica and J. Sou\u{c}ek}, 
``Cartesian Currents in the Calculus of Variations II. Variational Integrals'',
Ergebnisse der Mathematik und ihrer Grenzgebiete, Vol. 38,
Springer-Verlag, Berlin-Heidelberg, 1998.


	\bibitem{Giusti:84}
{E. Giusti},
\newblock{``Minimal Surfaces and Functions of Bounded Variation''},
\newblock Birkh\"auser, Boston (1984).


\bibitem{Ham}
{C. Hamburger}, \textit{Some Properties of the Degree for a Class of Sobolev Maps}, Proceedings: Mathematical, Physical and Engineering Sciences {\bf 455} (1999), 2331--2349.



\bibitem{JS}
 { R.L. Jerrard and H.M. Soner}, \textit{Functions of Bounded Higher Variation}, Indiana Univ. Math.  J. \textbf{51}(3) (2002), 645--677.

\bibitem{Mu}
S. M\"uller, T. Qi and B. S. Yan, 
{\it On a new class of elastic deformations not allowing for
	cavitation}, Inst. H. Poincar\'e Anal. Non Lin\'eaire {\bf 11} (1994), 217--243.


%
%
\bibitem{Miru}
P. Mironescu, {\it Sobolev maps on manifolds: degree, approximation, lifting}. Perspectives in nonlinear partial differential
equations, Contemp. Math. {\bf 446}
(2007), 413--436. 


\bibitem{Pa}
E. Paolini, {\it On the relaxed total variation of singular maps}. Manuscripta
Math. \textbf{111}(4) (2003), 499--512.

\bibitem{PVS} A. C. Ponce and  J. Van Schaftingen, {\it Clousure of smooth maps in $W^{1,p}(B^3,S^2)$},
Differential Integral Equations, {\bf 22} (9-10) (2009), 881--900.


\bibitem{S}
R. Scala, {\it Optimal estimates for the triple junction function and other surprising aspects of the
area functional}, Ann. Sc. Norm. Super. Pisa Cl. Sci. (5) {\bf 20}(2), 491-564. (2020).

\end{thebibliography}
\end{document}